\documentclass[12pt]{article}
\usepackage[utf8]{inputenc}
\usepackage[legalpaper, margin=1.2in]{geometry}
\usepackage[dvipsnames]{xcolor}

\usepackage{times}
\usepackage{epsfig}
\usepackage{graphicx}
\usepackage{amsmath}
\usepackage{amssymb}
\usepackage{tabularx} 
\usepackage{mathtools} 
\usepackage{amsfonts}
\usepackage{amsthm}
\usepackage{booktabs}\usepackage{enumitem}
\usepackage{authblk}
\usepackage{tikz-cd}
\usepackage{tikz}
\usetikzlibrary{arrows}

\usepackage[ruled,vlined,linesnumbered]{algorithm2e}
\usepackage[sort]{cite}
\usepackage{microtype}
\usepackage{appendix}
\usepackage{babel}
\addto{\captionsenglish}{}

\addto\captionsspanish{

}

\usepackage{chngcntr}
\usepackage{apptools}
\AtAppendix{\counterwithin{theorem}{section}}

\usepackage[pagebackref=true,breaklinks=true,colorlinks,bookmarks=false]{hyperref}

\usepackage{tabularx} 
\usepackage{amsmath}
\usepackage{graphicx} 
\usepackage{amsmath}
\usepackage{mathtools} 
\usepackage{amsfonts}
\usepackage{amssymb}
\usepackage{amsthm}

\usepackage{enumitem}
\usepackage{dsfont}
\usepackage{tikz-cd} 
\usepackage{comment}
\usepackage[utf8]{inputenc}
\usepackage{changepage}
\usepackage[format=plain, font=footnotesize]{caption}
\usepackage[sort]{cite}

\hypersetup{
	colorlinks=true,       
	linkcolor=teal,        
	citecolor=orange,        
	filecolor=magenta,     
	urlcolor=blue         
}
\usepackage{blindtext}

\newcommand{\rank}{\mathrm{rank}\:}

\newcommand{\Gr}{\mathrm{Gr}}
\newcommand{\PP}{\mathbb{P}}
\newcommand{\RR}{\mathbb{R}}

\newcommand{\CC}{\mathbb{C}}

\newcommand{\Ca}{\textit{\textbf C}}

\usepackage[capitalise, noabbrev, nameinlink]{cleveref}

\theoremstyle{plain}

\newtheorem{theorem}{Theorem}[section] 
\newtheorem{proposition}[theorem]{Proposition}
\newtheorem{lemma}[theorem]{Lemma}
\newtheorem{conjecture}[theorem]{Conjecture}

\crefname{conjecture}{Conjecture}{Conjectures}

\newtheorem{example}[theorem]{Example}
\newtheorem{remark}[theorem]{Remark}
\AtBeginEnvironment{remark}{%
  \pushQED{\qed}%
}
\AtEndEnvironment{remark}{\popQED\endexample}

\theoremstyle{definition}
\newtheorem{definition}[theorem]{Definition}

\newcommand{\subject}[1]{\gdef\thesubject{#1}}
\newcommand{\keywords}[1]{\gdef\thekeywords{#1}}

\subject{}

\makeatletter
\renewcommand{\maketitle}{%
  \begin{center}
    {\LARGE\bfseries \@title \par}
    \vskip 1em
    {\large \@author \par}
    \vskip 1em
    {\@date \par}
  \end{center}
  \vskip 1em
  \begingroup
    \renewcommand\thefootnote{}
    \ifx\thesubject\@empty\else \footnotetext{\textbf{2020 Mathematics Subject Classification. } \thesubject}\fi
    \ifx\thekeywords\@empty\else \footnotetext{\textbf{Keywords and Phrases. } \thekeywords}\fi
  \endgroup
  \setcounter{footnote}{0} 
}
\makeatother

\subject{14Q20, 14M15, 68T45.}

\begin{document}
\title{
Metric Multiview Geometry---a \\Catalogue in Low Dimensions
}
\author[1]{Timothy Duff}
\affil[1]{{\small University of Missouri, Columbia MO, USA}
\small 

\texttt{tduff@missouri.edu}

\url{https://orcid.org/0000-0003-2065-6309}}
\author[2]{Felix Rydell}
\affil[2]{{\small Swedish Defence Research Agency, Kista, Sweden}
\small 

\texttt{felix.rydell@foi.se}

\url{https://orcid.org/0000-0003-0300-8115}}
\date{}
\maketitle

\begin{abstract} We systematically compile an exhaustive catalogue of multiview varieties and anchored multiview varieties arising from projections of points, lines, and planes in 1, 2, and 3-dimensional projective spaces. We say that two such varieties are ED-equivalent if there is a linear isomorphism between them that preserves ED-critical points. This gives rise to fourteen representatives, at most two of which are equivalent, and we determine various properties—dimension, set-theoretic equations, and multidegrees—for all varieties featured in our catalogue. In the case of points, we also present a complementary study of resectioning varieties and their singular loci. Finally, we propose conjectures for the Euclidean distance degrees of all varieties appearing in our comprehensive compilation.

\end{abstract}
{
  \hypersetup{linkcolor=teal}
}


\section*{Introduction}

The \textit{Structure-from-Motion} pipeline in computer vision aims to create 3D computer models from 2D images captured by cameras with unknown parameters \cite{Hartley2004}. Given a set of $n$ 2D images, a typical implementation identifies \textit{correspondences} across the images that are recognizable as originating from the same 3D feature. These correspondences are used for camera \textit{calibration}, estimating the camera parameters, and subsequently \textit{triangulation}, which finds the 3D features that minimize the reprojection error. 

Multiview varieties are an essential concept in Structure-from-Motion. They provide mathematical models for the set of all possible image feature correspondences from a given set of known cameras. They are Zariski closures of images of rational maps that describe how world features are captured by cameras. The first multiview variety $\mathcal M_{\Ca}$ was formally defined in \cite{heyden1997algebraic} as the Zariski closure of the projection map
\begin{align}\label{eq: Phi C}\begin{aligned}
    \Phi_{\Ca}: \PP^3&\dashrightarrow (\PP^2)^n,\\
    X&\mapsto (C_1X,\ldots,C_nX),
    \end{aligned}
\end{align}
given a camera arrangement $\Ca=(C_1,\ldots,C_n)$ of full rank $3\times 4$ matrices $C_i$. The varieties $\mathcal{M}_{\mathbf{C}}$ are well-studied from both geometric and algebraic points of view---see e.g.~\cite{aholt2013hilbert,trager2015joint,li2018images,agarwal2019ideals}. 

In applications, it is important to know a set-theoretic description of $\mathcal M_{\Ca}\;$: assume that in the process of calibration, we are given the data of a point correspondence $x= (x_1,\ldots, x_n)$ in $(\PP^2)^n$. For the cameras $\Ca$ to be compatible with this data, we need $x\in \mathcal M_{\Ca}$. This puts constraints on $\Ca$, as the equations that cut out the multiview variety must vanish at $x$. When solving the triangulation problem, we are instead given a known camera arrangement $\Ca$ and a tuple of (noisy) image data $\widetilde{x}=(\widetilde{x}_1,\ldots,\widetilde{x}_n)$. The aim is to find the unique point $X\in \PP^3$ that best describes the image data. In practice, this is done by minimizing the so-called reprojection error in a choice of affine patch; see \cite{beardsley1994navigation,beardsley1997sequential,stewenius2005hard}. It works as follows. First, we find an approximation $x$ of $\widetilde{x}$ that lies on the multiview variety $\mathcal M_{\Ca}$. Second, we intersect all back-projected lines of $x_i$, for $i=1,\ldots,n$, to obtain $X$. From the algebraic point of view, the first part has been studied through \textit{Euclidean distance degrees} (EDDs) in specific cases \cite{harris2018chern,EDDegree_point,Rydell_2023_ICCV}. The EDD of an algebraic variety is the number of smooth complex critical points for the nearest point problem given generic data. It is a complexity measure for the corresponding optimization problem \cite{draisma2016euclidean}.

Various generalizations of \eqref{eq: Phi C} have been explored for different applications in the computer vision literature. Shashua and Wolf \cite{wolf2002projection} and Hartley and Vidal \cite{hartley2008perspective} examine projections $\mathbb{P}^N \dashrightarrow \mathbb{P}^2$ for the analysis of dynamic scenes. Recently, Rydell et al. \cite{Rydell_2023_ICCV} consider projections $\mathbb{P}^2 \dashrightarrow \mathbb{P}^1$ and $\mathbb{P}^1 \dashrightarrow \mathbb{P}^1$ for triangulation approaches that preserve incidence relations. These projections also appear in other works \cite{quan1997affine,faugeras1998self}. In addition, $\PP^3\dashrightarrow\PP^1$ is commonly used to model radial cameras \cite{thirthala2005multi,thirthala2005radial,hruby2023four}. Projecting lines from $\PP^3$ to $\PP^2$ instead of points leads to line multiview varieties, whose algebraic and geometric properties are described in \cite{breiding2023line,LMI}. The study of projections of higher dimensional subspaces was initiated in \cite{rydell2024projections}. In this direction, for a projective subspace $V$, let $\Gr(k,V)$ denote the \textit{Grassmannian of $k$-planes}, i.e., the set of $k$-dimensional subspaces of $V$. For a full rank matrix $C:\PP^N\dashrightarrow\PP^h$ and $P\in \Gr(k,\PP^N)$ spanned by $X_0,\ldots, X_k$, we define $C\cdot P\in \Gr(k,\PP^h)$ to be the span of $CX_0,\ldots, CX_k$. The \textit{(generalized) multiview variety} $\mathcal M_{\Ca,k}$ is the Zariski closure of the image of the rational map
\begin{align}\begin{aligned}
    \Phi_{\Ca,k}: \Gr(k,\PP^N)&\dashrightarrow \Gr(k,\PP^h)^n,\\
    P&\mapsto (C_1\cdot P,\ldots, C_n\cdot P),\end{aligned}
\end{align}
given a camera arrangement $\Ca=(C_1,\ldots,C_n)$ of full rank $(h+1)\times (N+1)$ matrices $C_i$.

For a Schubert variety (or more precisely an intersection of Schubert varieties) $\Lambda\subseteq \Gr(k,\PP^N)$, we further define the \textit{anchored} multiview variety $\mathcal M_{\Ca,k}^\Lambda$ as the Zariski closure of the image of $\Phi_{\Ca,k}$ restricted to $\Lambda$, written $\Phi_{\Ca,k}\restriction_\Lambda$. We say that $\mathcal M_{\Ca,k}^\Lambda$ is $\mathcal M_{\Ca,k}$ \textit{anchored at} $\Lambda$. See \Cref{fig:three-view} for an illustration. A key observation of \cite{Rydell_2023_ICCV} is that in the triangulation process, one can employ anchored multiview varieties to preserve incidence relations among point and line correspondences, while improving computational efficiency. As an example, let $\Ca$ be an arrangement of $n$ full rank $3\times 4$ cameras and consider a (noisy) line correspondence $\widetilde \ell=(\widetilde \ell_1,\ldots, \widetilde \ell_n)\in \Gr(1,\PP^2)^n$ across $n$ views and a noisy point correspondence $\widetilde x=(\widetilde x_1,\ldots,\widetilde x_n) \in (\PP^2)^n$. Assume the incidence conditions that $\widetilde x_i$ lies on $\widetilde \ell_i$ for each $i$. If we first triangulate the world point $X$ from the point correspondence, and then reconstruct the world line $L$ by  fitting the line correspondence $\widetilde \ell$ to the line multiview variety anchored at $\Lambda=\{L\in \Gr(1,\PP^3): X\in L\}$, then we guarantee that $X\in L$. Moreover, $\dim \Lambda< \dim \Gr(1,\PP^3)$, reducing the dimension of the optimization problem. This motivates us to study all possible anchored multiview varieties for $N=1,2,3$.

\begin{figure}
    \centering
\begin{tikzpicture}[scale=0.3]
\draw[fill] (1,5) circle (6pt);
\draw[fill] (11.5,1) circle (6pt);
\draw[fill] (22,5) circle (6pt);
\draw (1,8) -- (1,14);
\draw (1,8) -- (6,5);
\draw (6,5) -- (6,11);
\draw (1,14) -- (6,11)  node[above=0.45cm,right=-0.275cm] {};
\draw (8.5,3) -- (8.5,9);
\draw (8.5,3) -- (14.5,3);
\draw (14.5,3) -- (14.5,9);
\draw (14.5,9) -- (8.5,9) node[above=0.15cm, midway] {};
\draw (17,4) -- (17,10);
\draw (17,4) -- (22,7);
\draw (22,7) -- (22,13);
\draw (17,10) -- (22,13) node[above=0.1cm,left=0.75cm]{};

\fill[cyan!50, opacity=0.3] (1,8) -- (6,5) -- (6,11) -- (1,14) -- cycle;
\fill[cyan!50, opacity=0.3] (8.5,3) -- (14.5,3) -- (14.5,9) -- (8.5,9) -- cycle;
\fill[cyan!50, opacity=0.3] (17,4) -- (22,7) -- (22,13) -- (17,10) -- cycle;

\draw[line width=1.8pt, Apricot, opacity=10] (2.6,8.5) -- (4.4,8) node[black, above=0.35cm, right=-0.35cm, font=\bfseries] {};
\draw[line width=1.8pt, Apricot, opacity=10] (10.5,4.4) -- (12.7,4.8) node[black, above, midway, font=\bfseries] {};
\draw[line width=1.8pt, Apricot, opacity=10] (18.5,8) -- (22-5*0.29,5+14*0.29) node[black, above=0.2cm, right=-0.65cm, font=\bfseries] {};

\draw[dotted] (1,5) -- (11.375,18.4);

\draw[dotted] (11.5,1) -- (11.375,18.4);

\draw[dotted] (22,5) -- (11.375,18.4);



\draw (1,5) -- (1+10.375*0.23,5+13.4*0.23);
\draw (1+10.375*0.46,5+13.4*0.46) -- (11.375,18.4);
\fill[violet]  (1+10.375*0.241,5+13.4*0.241) circle (6pt);



\draw (11.5,1) -- (11.5-0.125*0.21,1+17.4*0.21);
\draw (11.5-0.125*0.46,1+17.4*0.46) -- (11.375,18.4);
\fill[violet] (11.5-0.125*0.205,1+17.4*0.205) circle (6pt);




\draw (22,5) -- (22-10.625*0.255,5+13.4*0.255);
\draw (22-10.625*0.405,5+13.4*0.405) -- (11.375,18.4);
\fill[violet] (22-10.625*0.255,5+13.4*0.255) circle (6pt);



\draw[line width=1.5pt, Apricot] (7,18) -- (17,19);
\node[Apricot] at (5.85,17.90) {$L$};


\fill[violet] (11.375,18.4) circle (8pt);
\node[violet] at (11.375,19.9) {$X$};

\end{tikzpicture}
    \hspace{1cm}
\begin{tikzpicture}[scale=0.42]
\draw[black,fill] (8,7) circle (4pt) node[below] {};
\draw[fill] (22,11) circle (4pt) node[below] {};




\draw[line width=1.25pt, violet] (12,18) -- (12-1.45*0.95,18+3.3*0.95); 
\draw[line width=1.25pt, violet] (12,18) -- (12+1.45*0.95,18-3.3*0.95); 




\fill[cyan!50, opacity=0.3] (6.75,15) -- (6.75,9) -- (12,8) -- (12,14) -- cycle;
\fill[cyan!50, opacity=0.3] (15.5,16) -- (15.5,10) -- (20.5,12) -- (20.5,18) -- cycle;

\draw[Apricot,fill] (12,18) circle (6pt) node[below] {};
\node[Apricot] at (10.85,18) {$X$};
\node[violet] at (13.5,16.85) {$L$};

\draw[dotted] (8,7) -- (10.62,21.14); 
\draw[dotted] (8,7) -- (13.38,14.87); 

\draw (8,7) -- (8+2.62*0.325,7+14.14*0.325);
\draw (8,7) --  (8+5.38*0.45,7+7.87*0.45);

\draw[line width=1.25pt, violet] (8+2.62*0.325,7+14.14*0.325) -- (8+5.38*0.45,7+7.87*0.45);

\draw (8+2.62*0.53,7+14.14*0.53) --  (10.62,21.14);
\draw (8+5.38*0.75,7+7.87*0.75) --  (13.38,14.87);







\draw[dotted] (22,11) -- (10.62,21.14);
\draw[dotted] (22,11) -- (13.38,14.87);

\draw (22,11) -- (22-11.38*0.325,11+10.14*0.325);
\draw (22,11) --  (22-8.62*0.57,11+3.87*0.57);

\draw[line width=1.25pt, violet] (22-11.38*0.325,11+10.14*0.325) --  (22-8.62*0.57,11+3.87*0.57);

\draw (22-11.38*0.52,11+10.14*0.52) -- (10.62,21.14);
\draw (22-8.62*0.75,11+3.87*0.75) --  (13.38,14.87);






\draw (6.75,15) -- (6.75,9);
\draw (6.75,9) --  (12,8);
\draw (12,8) -- (12,14);
\draw (12,14)-- (6.75,15);
\draw (15.5,16) -- (15.5,10);
\draw  (15.5,10) -- (20.5,12);
\draw (20.5,12) -- (20.5,18);
\draw (20.5,18) -- (15.5,16);




\fill[Apricot] (22-10*0.412,11+7*0.412) circle (6pt);
\fill[Apricot] (8+4*0.38,7+11*0.38) circle (6pt);

\end{tikzpicture}
    \caption{On the left: point projection $\PP^3\dashrightarrow (\PP^2)^3$ anchored at the line $L$. The image of $X$ lies on the image of $L$ in each camera plane. On the right: line projection $\Gr(1,\PP^3)\dashrightarrow \Gr(1,\PP^2)^2$ anchored at the set of lines through $X$. The image of $L$ contains the image of $X$ in each camera plane.}
    \label{fig:three-view}
\end{figure}

The structure and contributions of this paper are as follows. \Cref{s: Pre} fixes notation and terminology by recalling several standard facts. \Cref{s: VV} lists all (anchored) multiview varieties arising from projections from 1,2, and 3-dimensional projective spaces. \Cref{s: ED-equiv} builds on a key insight of \cite{Rydell_2023_ICCV} that some (anchored) multiview varieties are linearly isomorphic and that there is a natural bijection of critical points in the corresponding minimization of reprojection errors. When this happens, we say that those varieties are \textit{ED-equivalent}, and in \Cref{thm: ED-cat}, we reduce the catalogue to fourteen families, with only one possible equivalence remaining. Set-theoretic equations are described for a representative of each equivalence class in~\Cref{s: Set-Theoretic}, and we compute their multidegrees in \Cref{s: Multdeg}. \Cref{s:resectioning} initiates a parallel study of the resectioning varieties previously studied in~\cite{agarwal2022atlas,connelly2025algebra} and their singular loci (\Cref{thm:sing-res11}).
Finally, \Cref{s: conj} states conjectural Euclidean distance degrees for all (anchored) multiview varieties and resectioning varieties studied in this paper, based on computations in \texttt{Julia} \cite{bezanson2012julia,breiding2018homotopycontinuation} and \texttt{Macaulay2} \cite{grayson2002macaulay2}.

\bigskip

\paragraph{Acknowledgments.} Timothy Duff was supported by an NSF Mathematical Sciences Postdoctoral Research Fellowship (DMS-2103310). Felix Rydell was supported by the Knut and Alice Wallenberg Foundation within their WASP (Wallenberg AI, Autonomous
Systems and Software Program) AI/Math initiative.

\paragraph{AI Statement:} The initial version of this manuscript was prepared without any AI assistance. For the final version, we have used ChatGPT (OpenAI) and Claude (Anthropic) to proofread the manuscript, suggest clarifications and corrections, and to check, restructure and document the ancillary code. The manuscript is written solely by the authors, who take full responsibility for the paper's contents.


\section{Preliminaries}\numberwithin{equation}{section}\setcounter{equation}{0}\label{s: Pre} We collect the tools we use for the convenience of the reader. In \Cref{ss: wedge}, we define wedge products for vectors and matrices and relate them to the Plücker embedding. In \Cref{ss: smooth quad}, we establish classic results on smooth quadrics. In \Cref{ss: Cayley}, we consider the algebraic matrix group $\mathrm{SO}_n$ and its parametrization that we use for the proof of our main theorem in \Cref{s: ED-equiv}. In \Cref{ss: Euler}, we discuss the topological Euler characteristic, which we use in \Cref{s: ED-equiv} to prove that certain multiview varieties are non-isomorphic.

Throughout this paper, we always work over the complex numbers, and use the following notation. Fix $N\in \{0,1,2,\ldots\}$. A \textit{$k$-plane} is a $k$-dimensional subspace of $\PP^N$, the $N$-dimensional complex projective space.  We write $\Gr(k,\PP^N)$ for the Grassmannian of $k$-planes in $\PP^N$. $0$-planes are therefore points, $1$-planes are lines and $2$-planes are planes. In this paper, an \textit{isomorphism} of varieties refers to a regular (globally-defined) map with a regular inverse.
A linear isomorphism $C: \PP^N \to \PP^N$ is called a \textit{homography}.


\subsection{Wedge and cross products}\label{ss: wedge}
The set of all $k$-planes $P$ of $\PP^N$ is given the structure of an algebraic variety called the Grassmannian $\Gr(k,\PP^N)\subseteq \PP^{{N+1\choose k+1}-1}$, through the \textit{Plücker embedding} \cite[Section 8]{Gathmann}. If $P$ is spanned by $X_0,\ldots,X_k$, then the image of $P$ under the Plücker embedding is the vector $\iota(X_0,\ldots,X_k)$ of all ${N+1\choose k+1}$-many $(k+1)\times (k+1)$ minors of 
\begin{align}
    \begin{bmatrix}
        X_0 & \cdots & X_k
    \end{bmatrix}. \label{eq:xmat}
\end{align}
This yields a rational map $\iota : \left(\PP^N\right)^{k+1} \dashrightarrow \PP^{{N+1\choose k+1}-1}$ whose image is $\Gr(k,\PP^N)$. Note that $\iota$ is well-defined only for $(X_0, \ldots , X_k)$ such that the matrix~\eqref{eq:xmat} is full rank. The Grassmannian $\Gr(k,\PP^N)$ is isomorphic to $\Gr(N-k-1,\PP^N)$. In particular, $\Gr(1,\PP^2)\cong \PP^2$ and $\Gr(2,\PP^3)\cong \PP^3$.

Let $C:\PP^N\dashrightarrow \PP^h$ be a full rank linear map with $h\ge k$. We define $C\cdot P$ to be the $k$-plane spanned by $CX_0,\ldots,CX_k$. There is a ${h+1\choose k+1}\times {N+1\choose k+1}$ matrix $\wedge^{k+1} C$ that satisfies
\begin{align}
\iota(CX_0,\ldots,CX_k)=\wedge^{k+1} C\;\iota(X_0,\ldots,X_k). 
\end{align}
By construction, $\wedge^{k+1} I=I$, and for two matrices $C$ and $D$ such that $CD$ is well-defined, $\wedge^{k+1}(CD)=(\wedge^{k+1}C )\; (\wedge^{k+1}D)$. Next, we sketch an argument for why $\wedge^{k+1}C$ is full rank if $C$ is. Assuming that $h\le N$, the image of the mapping $C\cdot P$ equals $\Gr(k,\PP^h)$, and we can choose $X_i$ such that $\iota(CX_0,\ldots,CX_k)$ is any given unit vector. Then the span of $\Gr(k,\PP^h)$ is all of $\PP^{{h+1\choose k+1}-1}$, and $\wedge^{k+1}C$ has to be full rank. If $h\ge N$, we can let $C^\dagger$ be a pseudo-inverse satisfying $C^\dagger C=I$, and note that $(\wedge^{k+1}C^\dagger)\;(\wedge^{k+1}C)=I$ by the above. Since $\wedge^{k+1}C^\dagger$ is full rank, so must $\wedge^{k+1}C$ be.

For lines in $\PP^3$, the Plücker map $\iota $ may be identified as follows: let $X,Y\in\mathbb P^3$ and denote $X\wedge Y:=XY^\top - YX^\top$. The $4\times 4$ matrix $X\wedge Y$ is skew-symmetric and its upper triangular entries are the six $2\times 2$ minors of the $4\times 2$ matrix $\begin{bmatrix}
    X & Y
\end{bmatrix}$, i.e., up to a natural isomorphism, $X\wedge Y$ agrees with $\iota (X,Y)$. For lines in $\PP^2$, we instead apply the cross product $\times$. If $X,Y\in\mathbb P^2$ are distinct, then $X\times Y$ is the element of the dual $(\PP^2)^*$ that defines the unique linear equation that vanishes on the line spanned by $X$ and $Y$. In particular, $X\times Y$ equals $\iota (X, Y)$ up to left-multiplication by $\left[\begin{smallmatrix} 0 & 0 & 1\\ 0 & -1 & 0\\ 1 & 0 & 0
        \end{smallmatrix}\right]$.


\subsection{Smooth quadrics in $\PP^3$}\label{ss: smooth quad} As demonstrated by \cite{breiding2023line,LMI}, smooth quadrics are essential for the study of line multiview varieties. A smooth quadric $Q$ in $\PP^3$ is the set of $X\in \PP^3$ satisfying 
\begin{align}
    X^\top MX=0,
\end{align}
for a full rank $4\times 4$ matrix $M$, that we can assume to be symmetric. It is easy to check that this variety is smooth; the gradient $2MX$ is non-zero for every $X\in \PP^3$. 

\begin{example} A canonical choice of smooth quadric $Q$ is the image of the Segre embedding
\begin{align}\begin{aligned}
    \sigma:\PP^1\times \PP^1&\to  \PP^3,\\
(a,b)&\mapsto(a_0b_0:a_0b_1:a_1b_0:a_1b_1).\end{aligned}
\end{align}
This map is an isomorphism onto $Q$, which is the quadric surface in $\PP^3$ given by the equation 
\begin{align}
   X_0X_3-X_1X_2 =0.
\end{align}
It contains two 1-dimensional families of lines. These are parametrized by
\begin{align}
    \ell_1(a):=\iota\Big(\begin{bmatrix}
       a_0&
       0 &       a_1 & 0 \\        0 &
       a_0&
        0&  a_1
    \end{bmatrix}^\top\Big)\quad\textnormal{ and } \quad  \ell_2(b):=\iota\Big(\begin{bmatrix}
         b_0 &
       b_1 &
       0 & 0 \\
          0 &
        0&
       b_0& b_1
    \end{bmatrix}^\top\Big).
\end{align}
\end{example}

In general, we have the following well-known results:

\begin{theorem} A smooth quadric in $\PP^3$ contains two 1-dimensional families of lines. Any two lines in the same family are disjoint, and any two lines from different families meet in exactly one point.   
\end{theorem}

\begin{theorem} To three pairwise disjoint lines in $\PP^3$, there is a unique smooth quadric $Q$ that they are contained in. To four generic lines in $\PP^3$, there are exactly two lines that meet all of them. 
\end{theorem}

For more information, see the lecture notes by Stevens \cite[Lemma 6.16]{stevensmma320} and the textbook by Eisenbud and Harris \cite[Sections 3.4 and 3.6]{eisenbud-harris:16}. See also \cite[Section 1.2]{breiding2023line} for a survey on these results.

\begin{example}
We show how to parametrize the families of lines contained in a smooth quadric $Q$ defined by fixed pairwise disjoint lines $L_i:=\iota(U_i,V_i)$ for $i=1,2,3$ in $\PP^3$. Denote by $\Lambda$ the variety of lines in $\PP^3$ meeting each $L_i$ in a point. Let $\underline{U}_i,\underline{V}_i\in \CC^4$ be fixed affine representatives of $U_i,V_i\in \PP^3$. For $a=(a_0,a_1)\in \CC^2$, let $h_a\in \CC^4$ be the vector $(\det A_1,-\det A_2,\det A_3,-\det A_4)$, where 
\begin{align} A:=\begin{bmatrix}
     a_0\underline{U}_1+a_1\underline{V}_1 & \underline{U}_2 & \underline{V}_2
\end{bmatrix}^\top
\end{align}
and $A_i$ is $A$ with column $i$ removed. By construction, $h_a$ spans the kernel of $A$ and defines the hyperplane in $\PP^3$ spanned by $ a_0\underline{U}_1+a_1\underline{V}_1$ and $L_2$. Further, $h_a$ is linear in $a$, meaning that $h_a=\begin{bmatrix} h_{(1,0)} & h_{(0,1)}
\end{bmatrix} a$. The plane $h_a$ meets $L_3$ in a unique projective point. In this direction, the equation
\begin{align}
h_a^\top(b_0\underline{U}_3+b_1\underline{V}_3)=0
\end{align}
has one solution in $(b_0:b_1)\in \PP^1$. This solution is $b(a)=(h_a^\top\underline{V}_3:-h_a^\top\underline{U}_3)$ and is also linear in $a\in \CC^2$. Then
\begin{align}\begin{aligned}
    \phi:\PP^1&\to \Lambda,\\
    (a_0:a_1)&\mapsto \iota(a_0\underline{U}_1+a_1\underline{V}_1,b_0(a)\underline{U}_3+b_1(a)\underline{V}_3),
\end{aligned}
\end{align}
is a projectively well-defined parametrization of $\Lambda$. By construction, each line $\phi(a)$ meets $L_1$ and $L_3$. To see that it meets $L_2$, note that $\phi(a)$ lies inside the plane defined by $h_a$, which also contains $L_2$. In this plane, these two lines must meet.

Further, $\phi(a)$ can be written $(B_\Lambda \circ \nu)(a)$ for a $6\times 3$ matrix $B_\Lambda$, and the Veronese embedding $\nu:\PP^1\to \PP^2$ sending $(a_0:a_1)$ to $(a_0^2:a_0a_1:a_1^2)$. The matrix $B_\Lambda$ is full rank, because the map $\phi$ is injective and there is no linear injective map from an irreducible conic (the image of $\nu$) to a line. As a consequence, $\Lambda$ is a degree-2 curve in $\PP^5$.
 \hfill$\diamondsuit\,$ 
\end{example}


\subsection{The Cayley parametrization}\label{ss: Cayley}

Consider the orthogonal and special orthogonal matrix groups of $n\times n$ matrices over a field $K$:
\begin{align}
    \mathrm{O}_n(K):= \{A: A^\top A=I\}\quad \textnormal{ and } \quad \mathrm{SO}_n(K):= \{A: A^\top A=I, \det A=1\}.  
\end{align}
In our paper, we consider $K=\RR,\CC$. For these fields, $\mathrm{O}_n(K)$ is reducible and $ \mathrm{SO}_n(K)$ is irreducible \cite[Section 5.2]{boij1995introduction}. As varieties, they are both of dimension ${n\choose 2}$ over $K$. For any real skew-symmetric matrix $S$, we have that $I-S$ is invertible. To see why, let $X$ be a real element of its kernel. It satisfies $X=SX$. Then $X^\top X=(SX)^\top X= X^\top S^\top X=-X^\top (SX)=-X^\top X$, implying that $X^\top X=0$ and $X=0$. The Cayley parametrization of $\mathrm{SO}_n(K)$ sends a real skew-symmetric matrix $S$ to 
\begin{align}
    O=(I+S)(I-S)^{-1}.
\end{align}
This map is injective as we can recover $S$ from $O$. Indeed, we have $S=(O+I)^{-1}(O-I)$. The set of skew-symmetric matrices is ${n\choose 2}$-dimensional and therefore this injective map is dominant onto $\mathrm{SO}_n(K)$. By \cite[Theorem 4.3]{breiding2021algebraic}, it follows that $\overline{\mathrm{SO}_n(\RR)}=\mathrm{SO}_n(\CC)$. Therefore, the matrices obtained by inserting real skew-symmetric
matrices $S$ into the Cayley parametrization form a Zariski-dense
subset of $\mathrm{SO}_n(\mathbb{C})$. For the rest of the paper, we always put $K=\CC$ and simply write $\mathrm{SO}_n$.

Assume first that $n\le m$. For a matrix
$B\in\CC^{m\times r}$ of rank at most $n$, define the
\textit{restricted Stiefel variety}
\begin{align}\begin{aligned}
\mathrm{St}(n,m;B)
:=
\left\{
A\in\CC^{n\times m}:
\begin{aligned}
AA^\top &= I_n,\\
(I_m-A^\top A)B &= 0
\end{aligned}
\right\}.
\end{aligned}\end{align}
If $\rank B=n$, these conditions are equivalent to
$AA^\top=I_n$ and $\mathrm{Im}\,A^\top=\mathrm{Im}\,B$. Suppose that $B$ is real and $\rank B=n$. Choose $n$
independent columns of $B$, forming a matrix
$Y\in\RR^{m\times n}$, and put
$s:=\sqrt{\det(Y^\top Y)}>0$. Define
\begin{align}
\mathrm{St}^{+}(B)= \mathrm{St}^{+}(n,m;B;Y)
:=
\{A\in\mathrm{St}(n,m;B):\det(AY)=s\}.
\end{align}
We often suppress $n$, $m$, and $Y$ from the definition. Its real points are precisely those satisfying $\det(AY)>0$. Choose $Q\in\mathrm{GL}_n(\RR)$ such that
$Q^\top Y^\top YQ=I_n$ and $\det Q>0$. Then
\begin{align}
\mathrm{St}^{+}(n,m;B;Y)
=
\{OQ^\top Y^\top:O\in\mathrm{SO}_n(\CC)\}.
\end{align}
Thus $\mathrm{St}^{+}(n,m;B;Y)$ is isomorphic to
$\mathrm{SO}_n(\CC)$ and is irreducible. The choice
of $Y$ specifies an orientation of $\mathrm{Im}\,B$.

For $n>m$, define instead
\begin{align}
\mathrm{St}(n,m;B)
:=
\{A\in\CC^{n\times m}:A^\top A=I_m\}.
\end{align}
In this case, the condition
$(I_m-A^\top A)B=0$ is automatic, and the variety
is irreducible. We therefore set
$\mathrm{St}^{+}(B) = \mathrm{St}^{+}(n,m;B):=\mathrm{St}(n,m;B)$.


\subsection{Euler characteristics}\label{ss: Euler}

There are many different ways of defining the (topological) Euler characteristic. For instance, if we are given a triangulation of a topological space $\mathcal N$, the Euler characteristic $\chi(\mathcal N)$ is the alternating sum
\begin{align}
    k_0-k_1+k_2-\cdots,
\end{align}
where $k_i$ is the number of simplices of dimension $i$. Here, for $n=0,1,2,\ldots,$ a \textit{simplex} is a polytope of dimension $n$ with $n+1$ vertices, and a \textit{triangulation} is essentially a way of writing a space as a union of simplices that intersect nicely. Importantly, all real and complex algebraic varieties can be triangulated with respect to Euclidean topology \cite{hofmann2009triangulation}. The Euler characteristic can also be defined via singular homology \cite[Chapter 2]{hatcher2005algebraic}. More generally, it is defined for sheaves \cite[Section 16]{Gathmann}\cite[Section 4]{maxim2019intersection}. It turns out that with respect to the Euclidean topology and the constant sheaf, this construction of the Euler characteristic coincides with that of singular homology \cite[Chapter 3]{bredon2012sheaf} \cite[Remark 2.5.12]{dimca2004sheaves}. 

Here, we collect some basic properties of Euler characteristics.

\begin{lemma}\label{le: Eul2} Let $f: \mathcal N\to \mathcal M$ be an isomorphism of varieties, then 
\begin{align}
    \chi(\mathcal N)=\chi(\mathcal M).
\end{align}
\end{lemma}
\begin{proof} By \cite[Section 2.1]{hatcher2005algebraic}, homeomorphisms between topological spaces preserve the Euler characteristic, and isomorphisms between varieties are homeomorphisms (with respect to both Euclidean and Zariski topologies).
\end{proof}

\begin{lemma}\label{le: Eul1} Let $\mathcal N,\mathcal M$ be complex varieties, affine or projective. Then 
\begin{align}
    \chi(\mathcal M\cup \mathcal N)=\chi(\mathcal M)+\chi(\mathcal N)-\chi(\mathcal M\cap\mathcal  N).
\end{align}
    \end{lemma}

\begin{proof} This is explained in \cite[Section 2]{cappell2008euler} and \cite[Section 7]{maxim2019intersection}. 
\end{proof}
    
\Cref{le: Eul1} does not hold over the real numbers. For instance, $\chi(\RR)=1$, while $\chi(\{x\})=1$ and $\chi(\RR\setminus \{x\})=2$.

\begin{lemma}\label{le: Eul3} The Euler Characteristic of $\Gr(k,\PP^N)$ is ${N+1\choose k+1}$. In particular, the Euler characteristic of $\PP^N$ is $N+1$. 
\end{lemma}

\begin{proof} An element $P$ of $\Gr(k,\PP^N)$ does not uniquely correspond to a set of spanning vectors $X_0,\ldots,X_k$. However, after Gaussian elimination we obtain a unique reduced row echelon form of the $(k+1)\times(N+1)$ matrix consisting of the rows $X_i^\top$. Such a reduced row echelon form has a $(k+1)\times (k+1)$ identity matrix as a submatrix. Each possible placement of this $I$ corresponds to a stratum of $\Gr(k,\PP^N)$, and each stratum is isomorphic to a power of $\CC$. These are contractible spaces of Euler characteristic $1$. There are ${N+1\choose k+1}$ placements of $I$, implying that the Euler characteristic equals this number. 

Setting $k=0$, we have $\Gr(k,\PP^N)=\PP^N$, and the second statement follows from the first. This is a standard result, for instance found in \cite[Section 10.1]{may1999concise}.
\end{proof}


\section{Varieties from Vision} \numberwithin{equation}{section} \label{s: VV} Numerous algebraic varieties have been proposed in the context of calibration and triangulation in computer vision. We focus on multiview varieties and anchored multiview varieties that are defined by projections from $\PP^1,\PP^2,$ and $\PP^3$. In this section, we formally define these varieties and provide a complete list of them. 

The notation that we use is as follows. A full rank matrix is called a \textit{camera}, and a camera $C$ of size $(h+1)\times (N+1)$ induces a rational map $\PP^N\dashrightarrow \PP^h$ sending $X$ to $CX$. We define a \textit{camera arrangement} to be a list $\Ca=(C_1,\ldots,C_n)$ of camera matrices $C_i$ of sizes $(h_i+1)\times (N+1)$ with $n\ge 1$. Given a camera matrix $C$, let $c:=\ker (C)$ denote its \textit{center}. For a nonnegative integer $k\le h_i$, we define the rational map $P\mapsto C\cdot P$ that sends a $k$-plane $P$ spanned by $X_0,\ldots,X_k$ to the $k$-plane spanned by $CX_0,\ldots,CX_k$ in Plücker coordinates. This map is well-defined precisely when $P$ does not intersect $c$, and we therefore always assume that $k\le N-\dim c-1$. Note that for $h=2$ and $k=1$, the map $\Gr(1,\PP^3)\dashrightarrow\Gr(1,\PP^2)$ is instead often defined via $CX_0\times CX_1$ rather than $CX_0\wedge CX_1$, where $\times$ is the cross-product \cite{breiding2023line,LMI}. As noted in \Cref{ss: wedge}, up to signed permutation, these maps are equivalent.

\subsection{Multiview Varieties} \label{ss: MV} 

 A camera arrangement $\Ca$ defines a map as follows:
\begin{align}\begin{aligned}\label{eq: jointmap}
    \Phi_{\Ca,k}: \Gr(k,\PP^N)&\dashrightarrow \Gr(k,\PP^{h_1})\times \cdots\times\Gr(k,\PP^{h_n}),\\
    P&\mapsto \;\:(C_1\cdot P,\ldots,C_n\cdot P).
\end{aligned}
\end{align}
The \textit{(photographic) multiview variety} $\mathcal{M}_{\Ca,k}$ is the Zariski closure of the image of $\Phi_{\Ca,k}$. For the purposes of this paper, we restrict to the assumption that all $h_i$ are the same, i.e., each camera is of the same size. In this direction, we introduce the specialized map
\begin{align}\begin{aligned}\label{eq: Phi h spec}
    \Phi_{\Ca,k}^{N,h}:\mathrm{Gr}(k,\PP^N)&\dashrightarrow \mathrm{Gr}(k,\PP^h)^n,\\
    P&\mapsto (C_1 \cdot P,\ldots, C_n\cdot P).
\end{aligned}
\end{align}
We write $\mathcal{M}_{\Ca,k}^{N,h}$ for the Zariski closure of the image of this map. Note that the camera arrangement $\Ca$ encodes both $N$ and $h$, however, for the sake of clarity, we often specify $N$ and $h$ in the notation. 

For a camera $C:\PP^N\dashrightarrow \PP^h$, with $h\ge N$, we have that $I=\wedge^{k+1} C^\dagger \wedge^{k+1}C$, i.e., the image of $C$ is isomorphic to the image of the identity camera $I:\PP^N\to \PP^N$. This motivates our focus on the case $h\le N$. Further, if $h=k$, then the image of $\wedge^{k+1} C$ is a single point, which explains our restriction $h>k$. The tuples $(N,h,k)$, whose associated multiview varieties are of interest to us, satisfy $3\ge N\ge h>k\ge 0$. These are:
\begin{align}\begin{aligned}\label{eq: all settings}
    (3,3,2), & & (3,3,1),  & & (3,3,0), & & (3,2,1),  & & (3,2,0), \\
    (3,1,0), & & (2,2,1), & & (2,2,0), & & (2,1,0), & & (1,1,0).
\end{aligned}
\end{align}
In an attempt to reduce the number of subscripts, and make the notation more readable, we write 
\begin{align}
\mathcal{M}_{\Ca}^{N,h}:=\mathcal{M}_{\Ca,0}^{N,h},\quad  \mathcal{L}_{\Ca}^{N,h}:=\mathcal{M}_{\Ca,1}^{N,h}\quad \textnormal{and}\quad \mathcal{P}_{\Ca}^{N,h}:=\mathcal{M}_{\Ca,2}^{N,h}.
\end{align}
Further, we write $\mathcal M_k^{N,h}$ or $\mathcal M^{N,h},\mathcal L^{N,h},\mathcal P^{N,h},$ for $k=0,1,2,$ respectively, for the \textit{family} of all multiview varieties $\mathcal M_{\Ca,k}^{N,h}$ for arrangements $\Ca$ of cameras of size $(h+1)\times (N+1)$. Then \eqref{eq: all settings} corresponds to the following finite list of 10 families of multiview varieties:
\begin{proposition} The families of multiview varieties arising from \eqref{eq: Phi h spec} satisfying 
\begin{align}
    3 \ge N\ge h > k \ge 0
\end{align}
are
\begin{enumerate}[label=(\alph*)]
  \item $\mathcal{M}^{3,3}$, $\mathcal{M}^{3,2}$, $\mathcal{M}^{3,1}$, $\mathcal{M}^{2,2}$, $\mathcal{M}^{2,1}$, $\mathcal{M}^{1,1}$.\label{eq: enum 11}
    \item $\mathcal{L}^{3,3}$, $\mathcal{L}^{3,2}$, $\mathcal{L}^{2,2}$.\label{eq: enum 22}
    \item $\mathcal{P}^{3,3}$.\label{eq: enum 33}
\end{enumerate}
\end{proposition}

To be clear, $\mathcal{M}^{3,2}$ is the family of standard multiview varieties, the Zariski closures of maps \eqref{eq: Phi C}, and $\mathcal{L}^{3,2}$ is the family of standard line multiview varieties from \cite{breiding2023line}. Further, $\mathcal{M}^{3,1}$ is studied for radial cameras \cite{hruby2023four}, and $\mathcal{M}^{2,1}$, $\mathcal{M}^{1,1}$ are applied for the triangulation of correspondences with incidence relations \cite{Rydell_2023_ICCV}.


\subsection{Anchored Multiview Varieties}\label{ss: AMV}

Let $V=(V_1,\ldots,V_s)$ be a list of subspaces of $\PP^N$, and let $\lambda=(\lambda_1,\ldots,\lambda_s)$ be a list of nonnegative integers. 
For our purposes, a \textit{Schubert variety}\footnote{Technically, $\Lambda$ is an intersection of Schubert varieties.} of type $(V,\lambda)$ is a set 
\begin{align}
   \Lambda=\Lambda_{k}^N(V,\lambda):=\{P\in \mathrm{Gr}(k,\PP^N): \dim V_i\cap P\ge \lambda_i \textnormal{ for each } 1\le i \le s\}.
\end{align}
We assume $V$ and $\lambda$ are reduced in the sense that any sublists $V'$ and $\lambda'$ produce a larger variety $\Lambda$. This entails that $\dim V_i-\lambda_i<N-k$ for each $i$.

The \textit{(anchored) multiview variety} $\mathcal M_{\Ca,k}^{\Lambda,N,h}$ is the closure of the image of $\Phi_{\Ca,k}^{N,h}$, restricted to the Schubert variety $\Lambda$. In symbols,
\begin{align}\label{eq: anch closure}
   \mathcal M_{\Ca,k}^{\Lambda,N,h}:=\overline{\mathrm{Im}\;\Phi_{\Ca,k}^{N,h}\restriction_{\Lambda}}.
\end{align}
We say that $\mathcal M_{\Ca,k}^{\Lambda,N,h}$ is $\mathcal M_{\Ca,k}^{N,h}$ \textit{anchored at} $\Lambda$. We assume that the centers $c_i$ of the cameras $C_i$ in $\Ca$ meet any subspace $V_i$ in the smallest possible dimension, and we only consider irreducible Schubert varieties. Our first important observation is the next result.

\begin{lemma}\label{le: N N} $\mathcal M_{\Ca,k}^{\Lambda,N,N}$ is linearly isomorphic to $\Lambda$.
\end{lemma}

With this lemma, we do not mean to say that the anchored multiview varieties $\mathcal M_{\Ca,k}^{\Lambda,N,N}$ are trivial or not interesting to study in their own right; their Euclidean distance degrees cannot be deduced simply by the fact that they are isomorphic to $\Lambda$. 

\begin{proof} Following \Cref{ss: wedge}, for an invertible matrix $C:\PP^N\to \PP^N$, the induced map $\wedge^{k+1} C:\Gr(k,\PP^N)\to\Gr(k,\PP^N)$ is linear and full rank. Since $\wedge^{k+1}C$ is a square matrix, it must be invertible. Restricted to $\Lambda$, it is an isomorphism onto its image. 
\end{proof}

Note that for $ \Lambda_{k}^N(V,\lambda)$ to be nonempty, we must have that $\lambda_i\le \min\{\dim V_i, k\}$ for each $i$. If $V_i=\PP^N$, then $\dim V_i\cap P\ge \lambda_i$ is satisfied for any $P$; we may assume $\dim V_i<N$ for each $i$. 

Approaching a complete list of positive-dimensional nontrivial anchored multiview varieties arising from projections from $\PP^1,\PP^2,$ and $\PP^3$, we list all irreducible, positive-dimensional and proper Schubert varieties for $N=1,2,3,$ and $0\le \lambda_i\le k<N$. For $\Lambda_k^N(V,\lambda)$, let $\Omega$ denote the set of Schubert varieties we get by varying $V_i$. The associated \textit{families} of multiview varieties $\mathcal M_k^{\Omega,N,h}$ or $\mathcal M^{\Omega,N,h}$, $\mathcal L^{\Omega,N,h}$, $\mathcal P^{\Omega,N,h},$ for $k=0,1,2,$ respectively, consist of all multiview varieties $\mathcal M_{\Ca,k}^{\Lambda,N,h}$ for $\Lambda\in \Omega$ and a camera arrangement $\Ca$ with cameras of size $(h+1)\times (N+1)$. With this notation, we enumerate all types of Schubert varieties and all of their corresponding families of multiview varieties below.
\begin{enumerate}[leftmargin =4.5em]
    \item[$\bullet$ $k=0$:]  \label{enum: k = 0}    
    $N=2$:
    \begin{enumerate}
        \item[--] Points contained in a fixed line. We write $\mathcal M^{\mathrm{line},2,2},\mathcal M^{\mathrm{line},2,1}$ for the associated families of anchored multiview varieties.
    \end{enumerate}
   $N=3$: 
   \begin{enumerate}
       \item[--] Points contained in a fixed line. We write $\mathcal M^{\mathrm{line},3,3},\mathcal M^{\mathrm{line},3,2},\mathcal M^{\mathrm{line},3,1}$ for the associated families of anchored multiview varieties.
       \item[--] Points contained in a fixed plane. We write $\mathcal M^{\mathrm{plane},3,3},\mathcal M^{\mathrm{plane},3,2},\mathcal M^{\mathrm{plane},3,1}$ for the associated families of anchored multiview varieties.
   \end{enumerate}
  
    \item[$\bullet$ $k=1$:]\label{enum: k = 1} $N=2$:
    \begin{enumerate}
        \item[--] Lines through a fixed point. We write $\mathcal L^{\mathrm{point},2,2}$ for the associated family of anchored multiview varieties.
    \end{enumerate}
    $N=3$:
    \begin{enumerate}
        \item[--] Lines through a fixed point. We write $\mathcal L^{\mathrm{point},3,3},\mathcal L^{\mathrm{point},3,2}$ for the associated families of anchored multiview varieties.
        \item[--] Lines meeting $i\in \{1,2,3\}$ number of fixed pairwise disjoint lines. We write $\mathcal L^{\mathrm{line},3,3},\mathcal L^{\mathrm{line},3,2},$ $\mathcal L^{\mathrm{2line},3,3},\mathcal L^{\mathrm{2line},3,2},\mathcal L^{\mathrm{3line},3,3},\mathcal L^{\mathrm{3line},3,2}$ for the associated families of anchored multiview varieties.
        \item[--] Lines contained in a fixed plane. We write $\mathcal L^{\mathrm{plane},3,3},\mathcal L^{\mathrm{plane},3,2}$ for the associated families of anchored multiview varieties.
        \item[--] Lines contained in a fixed plane through a fixed point in that plane. We write $\mathcal L^{\mathrm{flag},3,3},$ $\mathcal L^{\mathrm{flag},3,2}$ for the associated families of anchored multiview varieties.
    \end{enumerate}
    
    \item[$\bullet$ $k=2$:]\label{enum: k = 2} $N=3$:
    \begin{enumerate}
        \item[--] Planes through a fixed point. We write $\mathcal P^{\mathrm{point},3,3}$ for the associated family of anchored multiview varieties.
        \item[--] Planes containing a fixed line. We write $\mathcal P^{\mathrm{line},3,3}$ for the associated family of anchored multiview varieties.
    \end{enumerate}
\end{enumerate}

In summary, we have 12 different types of irreducible, positive-dimensional and proper Schubert varieties. To them, we have associated 23 positive-dimensional, proper anchored multiview varieties. Continuing the enumeration of \Cref{eq: enum 11,eq: enum 22,eq: enum 33}, we list them as follows:

\begin{proposition} The families of anchored multiview varieties arising from \eqref{eq: anch closure} satisfying that $\Lambda$ is positive-dimensional, proper and 
\begin{align}
    3\ge N\ge h > k\ge \lambda_i\ge 0
\end{align}
are
\begin{enumerate}[label=(\alph*)] \setcounter{enumi}{3}
    \item $\mathcal M^{\mathrm{line},3,3}$,  $\mathcal M^{\mathrm{plane},3,3}$, $\mathcal M^{\mathrm{line},3,2}$, $\mathcal M^{\mathrm{plane},3,2}$, $\mathcal M^{\mathrm{line},3,1}$, $\mathcal M^{\mathrm{plane},3,1}$, $\mathcal M^{\mathrm{line},2,2}$, $\mathcal M^{\mathrm{line},2,1}$.\label{eq: anch 11}
    \item $\mathcal L^{\mathrm{point},3,3}$, $\mathcal L^{\mathrm{line},3,3}$, $\mathcal L^{\mathrm{2line},3,3}$, $\mathcal L^{\mathrm{3line},3,3}$, $\mathcal L^{\mathrm{plane},3,3}$,  $\mathcal L^{\mathrm{flag},3,3},$ $\mathcal L^{\mathrm{point},3,2}$, $\mathcal L^{\mathrm{line},3,2}$, $\mathcal L^{2\mathrm{line},3,2}$, $\mathcal L^{\mathrm{3line},3,2}$, $\mathcal L^{\mathrm{plane},3,2}$, $\mathcal L^{\mathrm{flag},3,2}$,  $\mathcal L^{\mathrm{point},2,2}$.\label{eq: anch 22}
     \item $\mathcal P^{\mathrm{point},3,3}$, $\mathcal P^{\mathrm{line},3,3}$.\label{eq: anch 33}
\end{enumerate}
\end{proposition}

\Cref{eq: enum 11,eq: enum 22,eq: enum 33,eq: anch 11,eq: anch 22,eq: anch 33} together form a list of 33 multiview varieties. The main theorem of this paper states that (at most) 14 of them are distinct with respect to a certain relation, and this is the subject of \Cref{s: ED-equiv}.

We take a closer look at the Schubert varieties defined by three pairwise disjoint lines, and their associated anchored multiview varieties.

\begin{example}\label{ex: L^3} Let $\Lambda$ denote the Schubert variety of lines intersecting three pairwise disjoint lines. In \Cref{ss: smooth quad}, we provided a parametrization $\phi:\PP^1\to \Lambda$. For a full rank $4\times 4$ matrix $C$, $\wedge^2 C$ is a full rank $6\times 6$ matrix, and the image of $\wedge^2 C\circ \phi$ is a curve of degree-2 in $\Gr(1,\PP^3)$. 

Given a full rank $3\times 4$ matrix $C$, $\wedge^2 C$ is a full rank $3\times 6$ matrix. We show that the image of $\wedge^2 C\circ \phi$ is a degree-2 curve in $\Gr(1,\PP^2)$, under the assumption that the center $c$ is away from the smooth quadric $Q$ defined by $\Lambda$. It is not hard to check that the image is of dimension 1, and that it is at most degree-2. If the degree of this curve were 1, then it would be a line in $\Gr(1,\PP^2)\cong\PP^2$. Such a line corresponds to a set of image lines through some fixed point $x\in \PP^2$. The back-projected line of $x$ must meet each line of $\Lambda$. A back-projected line with this property exists exactly when the center lies in $Q$.
 \hfill$\diamondsuit\,$ 
\end{example}

\subsubsection{Schubert varieties that are linearly isomorphic to projective subspaces} \label{sss: Schub}

Among the 12 types of Schubert varieties $\Lambda$ listed in this section, most of them are linearly isomorphic to some projective subspace. Below we describe these linear isomorphisms, from $\PP^{N_1}$ to $\Lambda$, where $N_1$ is the dimension of $\Lambda$ and the ambient space of $\Lambda$ is $\PP^N$. The only Schubert varieties that are not linear are sets of lines meeting $i$ fixed pairwise disjoint lines. 

For $k=0$, we have the set of points (1) in a fixed plane or (2) in a fixed line. Explicitly, the linear isomorphisms are from $\PP^2$ in case (1) and from $\PP^1$ in case (2). They are given by $Y\to HY$ for some full rank matrix $H$ of appropriate size.

For $k=1$, we have the set of lines (1) through a fixed point, (2) contained in a fixed plane, or (3) contained in a fixed plane through a fixed point in that plane. The set of lines through a fixed point $X$ in $\PP^{N}$ is isomorphic to $\PP^{N-1}$. In the case of (1), the Schubert varieties are parametrized by $Y\to \iota(X,HY)$ for some full rank $(N+1)\times N$ matrix $H$ whose image is away from $X$. For case (3), consider the set of lines through a fixed point $X$ inside a fixed plane $P$. A parametrization sends $Y\in \PP^{1}$ to $\iota(X,HY)$ for some full rank $4\times 2$ matrix $H$ whose image is a line inside $P$ away from $X$. For case (2), the set of lines inside a fixed plane $P$ in $\PP^3$ can be parametrized as $Z\mapsto (\wedge^2 H )Z=\iota(HY_0,HY_1)$, where $Y_0,Y_1\in \PP^2$ satisfy $Z=\iota(Y_0,Y_1)\in \Gr(1,\PP^2)\cong \PP^2$, and $H$ is a matrix whose image is $P$.

For $k=2$, we have the set of planes (1) through a fixed point or (2) through a fixed line. The set of planes in $\PP^3$ through a fixed point $X$ is 2-dimensional and is parametrized by $(Y_0,Y_1)\mapsto \iota(X,HY_0,HY_1)$, where $Y_0,Y_1\in \PP^2$ and $H$ is a full rank $4\times 3$ matrix whose image is away from $X$. This map can be viewed as a linear map sending $Z=\iota(Y_0,Y_1)\in \PP^2$ to the join of $X$ and the line $(\wedge^2 H )Z$. The set of planes through a fixed line $L=\iota(X_0,X_1)$ is 1-dimensional and is parametrized by $Y\mapsto \iota(X_0,X_1,HY)$, where $Y\in \PP^1$ and $H$ is a full rank $4\times 2$ matrix whose image is away from $L$.

Observe that there cannot be any other linear isomorphisms between the Schubert varieties $\Lambda$ and the projective spaces listed above, other than what we have described in each case. To see this, let $\phi_i:\PP^{N_1}\to \Lambda$ for $i=1,2$ both be linear isomorphisms. Then $\phi_1^{-1}\circ \phi_2$ is a linear isomorphism of $\PP^{N_1}$ to itself. This implies that $\phi_1^{-1}\circ \phi_2=H'$ for an invertible matrix $H'$, and $\phi_2=\phi_1\circ H'$. Note that if $N_1=2$, then there is a $H''$ such that $\wedge^2 H''=H'$.


\section{ED-Equivalence Catalogue}\label{s: ED-equiv}

In this section, we define ED-equivalence for multiview varieties and describe distinct equivalence classes with respect to this relation. As a consequence of this work, the list of 33 multiview varieties from \Cref{s: VV} can be shortened to a list of 14 varieties for the purposes of determining set-theoretic equations, multidegrees, singular loci and so on. 

We next recall the definition of the Euclidean distance degree as introduced in \cite{draisma2016euclidean}. 
For a variety $\mathcal{X}\subseteq \RR^m$ and a point $u\in\RR^m$ outside the variety, we consider the problem of finding the closest point on $\mathcal X$ to $u$:
\begin{equation}\label{eq :ED_problem}
    \mathrm{minimize}\quad \sum_{i=1}^m(u_i-x_i)^2\quad \textnormal{subject to} \quad x\in \mathcal{X}\setminus \mathrm{sing}(\mathcal X),
\end{equation}
where $\mathrm{sing}(\mathcal X)$ is the singular locus of $\mathcal X$. This is called the \textit{Euclidean distance problem} and models the process of error correction and fitting noisy data to a mathematical model $\mathcal{X}$. The \textit{Euclidean distance degree} (EDD) of $\mathcal X$ is the number of complex solutions to the critical point equations associated to~\eqref{eq :ED_problem}, called \textit{ED-critical} points, for a given generic point $u\in \RR^m.$ The EDD is an estimate of how difficult it is to solve this problem by exact algebraic methods. For a variety in a product of projective spaces, such as (anchored) multiview varieties, we study the EDD of an affine patch (an affine patch of a multiprojective space is a product of affine patches in each factor). In practice, we typically choose the affine patch $x_0=1$ in each factor. We introduce the notation $\eta:={h+1\choose k+1}-1$ for the dimension of the projective space that $\Gr(k,\PP^h)$ is embedded in via the Plücker embedding, so that the ambient space of $\mathcal M_{\Ca,k}^{\Lambda,N,h}$ is $(\PP^{\eta})^n$, where $n$ is the number of cameras of $\Ca$.

Anchored multiview varieties $\mathcal M_{\Ca,k}^{\Lambda,N,h}$
generalize multiview varieties. Indeed, setting
$\Lambda=\Gr(k,\PP^N)$ gives
$\mathcal M_{\Ca,k}^{\Lambda,N,h}=\mathcal M_{\Ca,k}^{N,h}$.

For the metric statements below, each factor of an affine patch
is realized as an affine hyperplane in its homogeneous coordinate
space, endowed with the induced Euclidean metric. Linear maps act
directly on these representatives, without subsequent
renormalization. Orthogonal projections are taken factorwise.
With these conventions, we make the following definition.

\begin{definition}\label{def: ED equiv} We say that two (anchored) multiview varieties $\mathcal M_{\Ca_0,k_0}^{\Lambda_0,N_0,h_0}$ and $\mathcal M_{\Ca_1,k_1}^{\Lambda_1,N_1,h_1}$ are \textit{ED-equivalent} if
\begin{enumerate}[label=(\Alph*)]
\item The arrangements $\Ca_j=(C_{j,1},\ldots,C_{j,n})$ for $j=0,1$ have the same number of cameras.
\item There is a linear isomorphism $\phi: \Lambda_1\rightarrow \Lambda_0$.
    \item For $B(C_{0,i}):= \: \wedge^{k_0+1} C_{0,i}\circ \phi$, there is a linear map
    \begin{align}\begin{aligned}
        \varphi=(\varphi_1,\ldots,\varphi_n):(\PP^{\eta_{0}})^n&\dashrightarrow (\PP^{\eta_{1}})^n,\\
        (z_1,\ldots,z_n)&\mapsto (A_1z_1,\ldots, A_nz_n),
    \end{aligned}
    \end{align}
     where each $A_i\in \mathrm{St}(B(C_{0,i}))$, that restricts to an isomorphism between $\mathcal M_{\Ca_0,k_0}^{\Lambda_0,N_0,h_0}$ and $\mathcal M_{\Ca_1,k_1}^{\Lambda_1,N_1,h_1} $.
    \item We have $\wedge^{k_1+1}C_{1,i}=\varphi_i\circ \wedge^{k_0+1} C_{0,i}\circ \phi$ for each $i$.
    \item Given a generic affine patch $U_0\subseteq (\PP^{\eta_{0}})^n$, and any affine patch $U_1\subseteq (\PP^{\eta_{1}})^n$ containing $\varphi(\mathcal M_{\Ca_0,k_0}^{\Lambda_0,N_0,h_0}\cap U_0)$, we have
    \begin{align}
       \mathrm{EDD}(\mathcal M_{\Ca_0,k_0}^{\Lambda_0,N_0,h_0}\cap U_0)=\mathrm{EDD}( \mathcal M_{\Ca_1,k_1}^{\Lambda_1,N_1,h_1}\cap U_1).
    \end{align}
    \item $\varphi$ is a bijection between ED-critical points given generic data $u \in U_0$ and those given data $\mathrm{pr}_{U_1}\:(\varphi(u))\in U_1$, where $\mathrm{pr}_{U_1}$ denotes orthogonal projection onto $U_1$.
\end{enumerate}
\end{definition}

\begin{definition} 
We say that two families of (anchored) multiview varieties $\mathcal M_{k_0}^{\Omega_0,N_0,h_0}$ and $\mathcal M_{k_1}^{\Omega_1,N_1,h_1}$ are \textit{ED-equivalent} and write 
\begin{align}
    \mathcal M_{k_0}^{\Omega_0,N_0,h_0}\leftrightarrow \mathcal M_{k_1}^{\Omega_1,N_1,h_1}
\end{align}
if for any $n\ge 1$, the following holds: to each generic camera arrangement $\Ca_0$, generic $\Lambda_0\in \Omega_0$, generic linear isomorphism $\phi$, and some $A_i$, there is a generic camera arrangement $\Ca_1$ and generic $\Lambda_1\in \Omega_1$ such that $\mathcal M_{\Ca_0,k_0}^{\Lambda_0,N_0,h_0}$ and $\mathcal M_{\Ca_1,k_1}^{\Lambda_1,N_1,h_1}$ are ED-equivalent, given the linear isomorphisms $\phi$ and $\varphi=(A_1,\ldots,A_n)$ and vice versa. This is an equivalence relation. If this relation does not hold, the families are called \textit{ED-distinct}. 
\end{definition}

The notion of genericity requires irreducibility. In this direction, note that the set of camera matrices of a certain size is an irreducible variety. Also, each set of Schubert varieties $\Omega$ listed in \Cref{ss: AMV} is irreducible. The parameter spaces $\mathrm{St}^+(B)$ used in the constructions
below are irreducible, as explained in Section~1.3. The irreducibility of the linear isomorphisms $\phi$ is discussed later in \Cref{le: Omega}.

The importance of ED-equivalence is two-fold. First, by its definition, if one family of multiview varieties is ED-equivalent to another, then we may restrict our study of most properties to the latter variety. This is due to the linear isomorphism assumed in \Cref{def: ED equiv}. The second virtue of this concept was demonstrated in \cite{Rydell_2023_ICCV}. The idea is that $\mathcal M_{\Ca_0,k_0}^{\Lambda_0,N_0,h_0}$ and  $\mathcal M_{\Ca_1,k_1}^{\Lambda_1,N_1,h_1}$ may be embedded into different ambient spaces. Therefore, when computing critical points with numerical software, it may be faster to solve the Euclidean distance problem over one multiview variety rather than the other, even though the EDDs are the same. If this is the case, then we may solve the problem in the faster setting and translate the solution via the bijection $\varphi$ of critical points to the setting of the other.

The main theorem of this paper classifies all multiview varieties from \Cref{s: VV} into distinct equivalence classes with respect to ED-equivalence:

\begin{theorem}\label{thm: ED-cat} All 33 families of (anchored) multiview varieties from \Cref{s: VV} are ED-equivalent to one of the following 14 families of (anchored) multiview varieties:
\begin{enumerate}[label=(\Roman*)]
    \item\label{enum: M classes} $\mathcal{M}^{3,3}$, $\mathcal{M}^{3,2}$, $\mathcal{M}^{3,1}$, $\mathcal{M}^{2,2}$, $\mathcal{M}^{2,1}$, $\mathcal{M}^{1,1}$,
    \item\label{enum: L classes} $\mathcal{L}^{3,3}$, $\mathcal{L}^{3,2}$, $\mathcal L^{\mathrm{line},3,3}$, $\mathcal L^{\mathrm{line},3,2}$, $ \mathcal L^{2\mathrm{line},3,3}$, $\mathcal L^{2\mathrm{line},3,2},$  $\mathcal L^{3\mathrm{line},3,2},\mathcal L^{3\mathrm{line},3,3}$.
\end{enumerate}
Among these 14 families, none (except possibly $\mathcal L^{3\mathrm{line},3,2}$ and $\mathcal L^{3\mathrm{line},3,3}$) are ED-equivalent.

Moreover,
\begin{enumerate}[label=(\roman*)]
    \item\label{enum: M33} $\mathcal{M}^{3,3}\leftrightarrow \mathcal{P}^{3,3}$,
    \item \label{enum: M22} $\mathcal{M}^{2,2}\leftrightarrow\mathcal{M}^{\mathrm{plane},3,3}$, $\mathcal{M}^{\mathrm{plane},3,2}$, $\mathcal{L}^{2,2}, \mathcal{L}^{\mathrm{point},3,3}$, $\mathcal{L}^{\mathrm{plane},3,3}$, $\mathcal{L}^{\mathrm{plane},3,2}$, $\mathcal{P}^{\mathrm{point},3,3}$, 
      \item\label{enum: M21}  $\mathcal{M}^{2,1}\leftrightarrow \mathcal{M}^{\mathrm{plane},3,1}$, $\mathcal{L}^{\mathrm{point},3,2}$,
    \item\label{enum: M11} $\mathcal{M}^{1,1}\leftrightarrow \mathcal{M}^{\mathrm{line},3,3}$, $\mathcal{M}^{\mathrm{line},3,2}$, $\mathcal{M}^{\mathrm{line},3,1}$, $\mathcal{M}^{\mathrm{line},2,2}$, $\mathcal{M}^{\mathrm{line},2,1}$, $\mathcal L^{\mathrm{flag},3,3}$, $\mathcal L^{\mathrm{flag},3,2}$, $\mathcal{L}^{\mathrm{point},2,2}$, $\mathcal{P}^{\mathrm{line},3,3}$.
\end{enumerate}
\end{theorem}


\begin{proposition}\label{prop: L^3} Let $\Lambda$ be the Schubert variety of lines meeting three fixed disjoint lines in $\PP^3$. The multiview variety $\mathcal L_{\Ca_0}^{\Lambda,3,3}$ is linearly isomorphic to $\mathcal L_{\Ca_1}^{\Lambda,3,2}$, where $\Ca_1$ is any camera arrangement with centers away from the smooth quadric $Q$ defined by $\Lambda$.
\end{proposition}

\begin{proof} Recall the map $\PP^1\to \Lambda$ that sends $a$ to $B_{\Lambda}\nu(a)$ from \Cref{ss: smooth quad}. Here, $B_{\Lambda}$ is a full rank $6\times 3$ matrix and $\nu$ is a Veronese embedding. Let $C_0$ be any $4\times 4$ camera and let $C_{1}$ be any $3\times 4$ cameras with center away from $Q$. The image of $\wedge^2 C_{1}$ restricted to $\Lambda$ is non-linear as in \Cref{ex: L^3}. This implies that $M:=(\wedge^2 C_{1} )B_{\Lambda}$ is rank-3, i.e., it is an invertible $3\times 3$ matrix. Then the linear map $(\wedge^2C_{0})B_{\Lambda}M^{-1}$ sends $(\wedge^2C_{1})L$ to $(\wedge^2C_{0})L$ for each $L\in \Lambda$. The inverse of this map is $\wedge^2\big( C_{1}C_{0}^{-1}\big)$. Given arrangements $\Ca_0$ and $\Ca_1$, the linear isomorphism acts on every factor as described above. 
\end{proof}

\subsection{Proof of the main theorem}

To prove \Cref{thm: ED-cat}, we need the help of a few lemmas. 

\begin{lemma}\label{le: ClosIsomo} Let $\psi:\mathcal X \to\mathcal Y$ be an isomorphism and let $U\subseteq \mathcal X,V\subseteq \mathcal Y$ be sets whose Euclidean closures equal their Zariski closures. If $\psi(U)=V$, then $\psi(\overline{U})=\overline{V}$.
\end{lemma}

\begin{proof} Take a point $v\in \overline{V}\setminus V$. Then there is a sequence $V\ni v^{(n)}\to v$ in Euclidean topology. By continuity of $\psi^{-1}$, $u^{(n)}:=\psi^{-1}(v^{(n)})\in U$ converges in Euclidean topology to a point $u\in \overline{U}$ for which $\psi(u)=v$. We have shown $\overline{V}\subseteq \psi(\overline{U})$. Similarly we show $\overline{U}\subseteq \psi^{-1}(\overline{V})$ from which it follows that $\psi(\overline{U})\subseteq \overline{V}$.
\end{proof}

\begin{proposition}\label{prop: EDD bij} Let
$\mathcal{X}_0, \mathcal{X}_1, \mathcal{Y}_0, \mathcal{Y}_1$ be irreducible varieties such that
\begin{align}
    \mathcal X_j\subseteq \PP^{N_j},\quad \mathcal Y_j\subseteq (\PP^{\eta_{j}})^n\quad \textnormal{ for }j=0,1.
\end{align}
Consider the diagram
\begin{equation}
    \begin{tikzcd}
\mathcal X_0 \arrow[r, dashed, "\phi", leftarrow] \arrow[d, dashed, "f_0"'] & \mathcal X_1 \arrow[d, dashed, "f_1"] \\
\mathcal Y_0 \arrow[r, dashed, "\varphi"] & \mathcal Y_1
\end{tikzcd}
\end{equation}
and assume that
\begin{itemize}
    \item $f_0$ is rational.
    \item $\phi$ is birational.
    \item $\varphi$ is an
isomorphism induced by
\begin{align}
(z_i)_i\longmapsto(A_i z_i)_i,
\qquad
A_i\in\CC^{(\eta_1+1)\times(\eta_0+1)},
\end{align}
where $(A_i^\top A_i z_i)_i = (z_i)_i$ on the affine cone of $\mathcal Y_0$.
    \item $f_1:= \varphi\circ f_0\circ \phi$.
\end{itemize}
Then $\varphi$ restricts to an isomorphism $\overline{\mathrm{Im}\; f_0}\to \overline{\mathrm{Im}\; f_1}$. Further, let $U_0\subseteq (\PP^{\eta_{0}})^n$ be an affine patch (a product of affine patches of each factor). The ED-critical points $x^*$ of $~ \overline{\mathrm{Im}\; f_0} \cap U_0$, given generic data $u\in U_0 $, are in bijection with the ED-critical points $y^*$ of $~ \overline{\mathrm{Im}\; f_1}\cap U_1$, given data $\operatorname{pr}_{U_1}\:(\varphi(u))$, where $U_1\subseteq (\PP^{\eta_{1}})^n$ is an affine patch containing $\varphi(\overline{\mathrm{Im}\; f_0}\cap U_0)$. This bijection is given by $x^*\mapsto y^*=\varphi (x^*)$. 
\end{proposition}

\begin{proof} To see that $\varphi$ restricts to an isomorphism, we use that the Euclidean and Zariski closures of images of rational maps coincide by Chevalley's theorem \cite[Theorem 3.16]{harris2013algebraic} and apply \Cref{le: ClosIsomo}. That $\varphi$ induces a bijection of critical points follows
from the assumption that $(A_i^\top A_i)_i$ acts as the identity
on the affine cone over $\mathcal Y_0$. Orthogonally projecting the transformed data onto $U_1$ changes
the squared distance function on $U_1$ only by an additive constant. The details are worked out in \cite[Appendix B.4]{Rydell_2023_ICCV}.
\end{proof}

\begin{lemma}\label{le: Omega} Fix an item among \Cref{enum: M33,enum: M22,enum: M21,enum: M11} in \Cref{thm: ED-cat}. Let $\mathcal M^{N_1,h_1}$ be the family of point multiview variety to the left of $\leftrightarrow$, and let $\mathcal M_{k_0}^{\Omega_0,N_0,h_0}$ be a family of (anchored) multiview variety on the right. The closure $\Sigma_0$ of the set of linear isomorphisms  $\phi:\PP^{N_1}\to \Lambda_0$, for $\Lambda_0\in \Omega_0$, is an irreducible variety. 
\end{lemma}

\begin{proof} First, we consider the non-anchored multiview varieties $\mathcal P^{3,3}$ and $\mathcal L^{2,2}$. In these cases, $\Omega_0$ contains only one Schubert variety. For $\mathcal P^{3,3}$, it is $\Gr(2,\PP^3)\cong \PP^3$ and for $\mathcal L^{2,2}$, it is $\Gr(1,\PP^2)\cong \PP^2$. Then any such linear isomorphism $\phi$ is an invertible linear map of the corresponding size. The set of such maps corresponds to an irreducible variety.

We laid the groundwork for the remaining cases in \Cref{sss: Schub}. We showed that for fixed $\Lambda_0$, the linear maps $\phi$ are parametrized by full rank matrices $H$ of fixed size that depend on $\Lambda_0$. Since the set of anchoring points, lines, or planes in each case is an irreducible variety, it follows that $\Sigma_0$ is irreducible.
\end{proof}

We divide the proof of \Cref{thm: ED-cat} into two parts for readability. The first one deals with the equivalences of \Cref{enum: M33,enum: M22,enum: M21,enum: M11}, and the other with the distinction of the families of \Cref{enum: M classes,enum: L classes}.

\begin{proof}[Proof of \Cref{thm: ED-cat}, Part 1]  Fix an item among \Cref{enum: M33,enum: M22,enum: M21,enum: M11}. Let $\mathcal M^{N_1,h_1}$ be the family of point multiview varieties to the left of $\leftrightarrow$ and let $\mathcal M_{k_0}^{\Omega_0,N_0,h_0}$ be a family of (anchored) multiview varieties on the right. Given a generic camera arrangement $\Ca_0$, and a generic $\Lambda_0\in \Omega_0$, consider $\mathcal M_{\Ca_0,k_0}^{\Lambda_0,N_0,h_0}$. We aim to use \Cref{prop: EDD bij}, and so in the notation of that result, we write $\mathcal X_1=\PP^{N_1}$ and $\mathcal X_0=\Lambda_0$. Then $\mathcal X_1$ and $\mathcal X_0$ are linearly isomorphic, and we let $\phi:\PP^{N_1}\to \Lambda_0$ denote a linear isomorphism. For instance, in \Cref{enum: M33}, the set of points in $\PP^3$ and the set of planes in $\PP^3$ are isomorphic via the identity map.

Let $C_{0,i}$ denote the $i$:th camera of $\Ca_0$, and define 
\begin{align}
\mathcal Y_0:=\overline{(\wedge^{k_0+1}C_{0,1})(\Lambda_0)}\times \cdots \times \overline{(\wedge^{k_0+1}C_{0,n})(\Lambda_0)}.
\end{align}
Recall that $\eta_0={h_0+1\choose k_0+1}-1$ refers to the dimension of $\PP^{\eta_0}$ that $\Gr(k_0,\PP^{h_0})$ is embedded in via the Plücker embedding. Write $B(C_{0,i}) =(\wedge^{k_0+1}C_{0,i})\circ\phi$. Let $A_i\in \mathrm{St}^+(B(C_{0,i}))$ be generic. Define
\begin{align}
    \mathcal Y_1:=\big(\overline{A_1\circ (\wedge^{k_0+1}C_{0,1})(\Lambda_0)}\big)\times \cdots \times \big(\overline{A_n\circ (\wedge^{k_0+1}C_{0,n})(\Lambda_0)}\big)=(\PP^{h_1})^n.
\end{align}
Then let $\Ca_1$ be the arrangements of cameras $C_{1,i}:=A_i\circ \wedge^{k_0+1}C_{0,i}\circ \phi$. Define $\varphi: (\PP^{\eta_0})^n\dashrightarrow (\PP^{\eta_1})^n$ by $z=(z_i)_i\mapsto (A_iz_i)_i$. Following \Cref{le: ClosIsomo}, $\varphi$ restricts to an isomorphism from $\mathcal Y_0$ to $\mathcal Y_1$. We can now apply \Cref{prop: EDD bij} to see that $\mathcal M_{\Ca_0,k_0}^{\Lambda_0,N_0,h_0}$ is ED-equivalent to $\mathcal M_{\Ca_1}^{N_1,h_1}$. 

In the following dominance argument, we use real matrix
representatives for the cameras and the domain isomorphism
$\phi$, and take complex Zariski closures of the resulting
parameter spaces. We fix a linear isomorphism $\phi:\PP^{N_1}\to \Lambda_0$. Define the affine incidence 
\begin{align}\begin{aligned}
\mathcal T_\phi :&=
\left\{
(A,C_0)
:
\rank C_0=h_0+1,\quad
A\in \mathrm{St}^+\bigl(B(C_0)\bigr)
\right\}\\
&\subseteq \RR^{(h_1+1)\times (\eta_0+1)}
\times \RR^{(h_0+1)\times(N_0+1)}.
\end{aligned}
\end{align}
Every irreducible component of the Zariski closure
of $\mathcal T_\phi$ dominates the camera space
under the projection $(A,C_0)\mapsto C_0$.

The proof is complete if we can show that the map 
\begin{align}\begin{aligned}
 \label{eq: map to be dom}
    \Psi: \mathcal T_\phi& \dashrightarrow\PP(\RR^{(h_1+1)\times (N_1+1)}),\\
    (A,C_0)&\mapsto A\circ \wedge^{k_0+1}C_{0}\circ \phi   
\end{aligned}
\end{align}
is dominant. Indeed, if the map is dominant, then the output of a generic input is generic, and the preimage of a generic $(h_1+1)\times (N_1+1)$ camera $C_{1}$ contains a generic pair $(A,C_{0})$ for which $C_{1}=A\circ \wedge^{k_0+1}C_{0}\circ \phi$. Since $A^\top A$ is the identity restricted to the image of $\wedge^{k_0+1}C_{0}\circ \phi$, we conclude $A^\top\circ C_{1}\circ \phi^{-1}= \wedge^{k_0+1}C_{0}\restriction_{\Lambda_0}$, which shows the ``vice versa'' part of the ED-equivalence. 

In the case that $k_0=0$, we can check by hand that $\Psi$ is dominant. Indeed, given generic $C_{0}$, it suffices to note that $C_{0}\circ \phi$ is a generic $(h_0+1)\times(N_1+1)$ matrix. Then multiplying with a full rank $(h_1+1)\times (h_0+1)$ matrix $A_i$ on the left, we get an arbitrary $(h_1+1)\times (N_1+1)$ matrix.

We verify dominance in the remaining cases by exact Jacobian
computations in \texttt{Julia}. Write
\begin{align}
F(A,C_0):=(I-A^\top A)B(C_0).
\end{align}
On the full-rank locus, $F(A,C_0)=0$ expresses the
image condition.

Near a test point $(A_*, C_*)$, we parametrize $A$ by
$A(S) = A_*(I+S)(I-S)^{-1}$ with $S$ skew-symmetric, as in the Cayley
parametrization, and $C_0$ by all its entries. In these
parameters, let $J_F$ denote the Jacobian of $F$, and let $J_\Psi$ denote the
Jacobian of $\Psi$ expressed in an affine chart of its projective target. At a smooth point
of the constrained parameter space, the rank of the restricted
differential is
\begin{align}
\rank (J_\Psi\restriction_{\ker J_F}) = \rank\!\begin{bmatrix}J_F\\J_\Psi\end{bmatrix}
-\rank J_F.
\end{align}
For each case, our computations produce a point at which this
rank equals $(h_1+1)(N_1+1)-1$, the dimension of the target.
Hence $\Psi$ is dominant. The \texttt{Julia} code implementing these exact Jacobian computations,
including the parametrizations, test points, and rank certificates,
is included as an ancillary file with the arXiv version of this paper.
\end{proof}

For part 2 of the proof, we first need a lemma.


\begin{lemma}\label{le: Sing} The Schubert variety $\Lambda$ of lines meeting a fixed line $L$ in $\PP^3$ has exactly one singular point, $L$ itself. 

As a consequence, any member of the family $\mathcal L^{\mathrm{line},3,3}$ is singular, and any member of the family $\mathcal L^{\mathrm{line},3,2}$ with at least two cameras with generic centers is singular. 
\end{lemma}

\begin{proof} We may take the line $L$ to be spanned by $(1:0:0:0)$ and $(0:1:0:0)$, because any other such Schubert variety differs by an invertible map $\PP^5\to \PP^5$. Then $\Lambda$ is the closure of the image of
\begin{align}\begin{aligned}
\PP^1\times \PP^2 &\dashrightarrow \PP^5, \\
   ((\lambda:\mu),(a:b:c)) &\mapsto  \iota\big((\lambda:\mu:0:0),(0:a:b:c)\big)\\
        &\quad\;=(
        \lambda a: \lambda b: \lambda c: \mu b:\mu c:0).
\end{aligned}
\end{align}
It is given by the equations $X_5=0$ and $X_1X_4-X_2X_3=0$. Since $X_1X_4-X_2X_3$ is an irreducible polynomial, the ideal generated by these two equations is radical. We have checked that the rank of the Jacobian drops precisely at the point $(1:0:0:0:0:0)\in \PP^5$, which represents the line $L$ in Plücker coordinates.  

Define $\Lambda$ to be the set of lines meeting $L$. By \Cref{le: N N}, $\mathcal L_{\Ca}^{\Lambda,3,3}$ is isomorphic to $\Lambda$. Since $\Lambda$ is singular, so is $\mathcal L_{\Ca}^{\Lambda,3,3}$. Regarding  $\mathcal L_{\Ca}^{\Lambda,3,2}$, note that the projection map given two cameras, whose centers together with $L$ span $\PP^3$, is injective around $L$. Further, for a Zariski open set containing $L$, the projection restricts to an isomorphism onto its image. Since isomorphisms send singular points to singular points, we are done.
\end{proof}

\begin{proof}[Proof of \Cref{thm: ED-cat}, Part 2] 
To see that all classes of \Cref{enum: M classes,enum: L classes} are distinct (except possibly $\mathcal L^{3\mathrm{line},3,2}$ and $\mathcal L^{3\mathrm{line},3,3}$), we first note that those whose domains are of different dimensions must be in distinct classes, because there can be no birational map $\phi$ in those cases. Next, if $\mathcal M_{\Ca_0,k_0}^{\Lambda_0,N_0,h_0}$ and  $\mathcal M_{\Ca_1,k_1}^{\Lambda_1,N_1,h_1}$ are of the same class, then there must be a linear isomorphism between the images of $\wedge^{k_0+1}C_{0,i}:\Lambda_0\dashrightarrow \Gr(k_0,\PP^{h_0})$ and $\wedge^{k_1+1}C_{1,i}:\Lambda_1\dashrightarrow \Gr(k_1,\PP^{h_1})$, given generic cameras $C_{0,i}$ and $C_{1,i}$. This argument directly shows that the only possible equivalences are between families within the items of the following list:
\begin{enumerate}
    \item $\mathcal{M}^{3,3}$, $\mathcal L^{\mathrm{line},3,3}$.\label{enum: II 3 3}   \item $\mathcal{M}^{3,2}$, $\mathcal L^{\mathrm{line},3,2}$.\label{enum: II 3 2}
    \item $\mathcal{M}^{2,2}$, $\mathcal L^{2\mathrm{line},3,3}$, $\mathcal L^{2\mathrm{line},3,2}$.\label{enum: II 2 2}
    \item $\mathcal{M}^{1,1}$, $\mathcal L^{3\mathrm{line},3,3}$, $\mathcal L^{3\mathrm{line},3,2}$.\label{enum: II 1 1}
\end{enumerate}
\Cref{enum: II 3 3,enum: II 3 2} consist of distinct classes, because $\mathcal L^{\mathrm{line},3,3},\mathcal L^{\mathrm{line},3,2}$ are singular by \Cref{le: Sing}, while $\mathcal{M}^{3,3}$ is isomorphic to $\PP^3$ by \Cref{le: N N} and therefore smooth, and $\mathcal{M}^{3,2}$ is smooth for $n\ge 3$ generic cameras \cite{trager2015joint}. 

For \Cref{enum: II 2 2}, we have by \Cref{le: N N} and \Cref{ss: Euler} that $\mathcal{M}^{2,2}$ is isomorphic to $\PP^2$ and has Euler characteristic 3, and that $\mathcal L^{2\mathrm{line},3,3}$ is isomorphic to $ \PP^1\times \PP^1$ and has Euler characteristic $4$. For arrangements $\Ca$ consisting of $n$ generic cameras, one can calculate the topological Euler characteristic $\chi(\mathcal L_\Ca^{2\mathrm{line},3,2})=4+2n$ as in \cite[Lemma D.1]{Rydell_2023_ICCV} or \cite{EDDegree_point} (for large enough $n$). Since Euler characteristics are preserved under isomorphisms, the three families of varieties in \Cref{enum: II 2 2} are ED-distinct.

We are left to investigate \Cref{enum: II 1 1}. Let $\Lambda$ be the Schubert variety of lines meeting three fixed lines. Let $C_0$ be a camera of size $3\times 4$ or $4\times 4$. By \Cref{ex: L^3}, the image $\wedge^2 C_0(\Lambda)$ is a degree-2 curve. However, the image $C_1(\PP^1)$ given a $2\times 2$ matrix $C_1$ is linear. Therefore there cannot be a linear isomorphism between the images, and $\mathcal M^{1,1}$ is ED-distinct from $\mathcal L^{3\mathrm{line},3,3}$ and $\mathcal L^{3\mathrm{line},3,2}$.
\end{proof}


\section{Set-Theoretic Equations}\label{s: Set-Theoretic}

In application, set-theoretic equations for multiview varieties are important as they provide constraints on the cameras, given data of image tuples, allowing for the recovery of camera parameters. For the complete study of set-theoretic equations among the (anchored) multiview varieties introduced in this paper, we may without restriction only consider the 14 varieties of \Cref{enum: M classes,enum: L classes}. This is a consequence of the fact that ED-equivalence entails linear isomorphism. In this section, we work with non-generic camera arrangements. However, we do assume that the Schubert varieties $\Lambda_k^N(V,\lambda)$ are generic in the sense that the subspaces of $V$ do not meet any camera centers. 

We work with back-projected planes, a standard approach in algebraic vision. Given an image subspace $p\in \Gr(k,\PP^h)$, its \textit{back-projected plane} in $\PP^N$ is the $(N-h+k)$-dimensional subspace of $k$-planes that project onto $p$. For instance, for the point projection $\PP^3\to\PP^3$, the back-projected planes are points. For the projection $\PP^3\to \PP^2$, the back-projected planes are lines and are aptly called \textit{back-projected lines} instead. For $\Gr(1,\PP^3)\dashrightarrow\Gr(1,\PP^2)$, the back-projected planes are planes.

We start by determining the dimensions of (anchored) multiview varieties.

\begin{proposition} For $N\in \{1,2,3\}$ and $h\le N$, let $\Ca$ be a camera arrangement. Then
\begin{align}
    \dim \mathcal M_{\Ca}^{N,h}= N
\end{align}
if and only if \textnormal{(1):} $(N,h)\in \{(3,3),(2,2),(1,1)\}$, or \textnormal{(2):} $(N,h)\in \{(3,2),(2,1)\}$ and there are two disjoint point centers, or \textnormal{(3):} $(N,h)\in \{(3,1)\}$ and there are three distinct line centers without a common intersection point.
\end{proposition} 

\begin{proof} General formulae for dimensions of point multiview varieties are known \cite{li2018images,rydell2024projections}. In this setting, we may use simpler arguments. 

\textnormal{(1):} Follows from \Cref{le: N N}.

\textnormal{(2):} In both cases, projecting a generic point $X$ onto two cameras of disjoint point centers leaves two back-projected lines that meet exactly at $X$. Then by the fiber dimension theorem, the dimension of the domain equals the dimension of the image.

\textnormal{(3):} The image of the associated projection map lies in $(\PP^1)^n$. Therefore, if $n=2$, the dimension cannot be 3. If $n\ge 3$ and all line centers meet in a point $X\in \PP^3$, let $Y\in \PP^3$ be a point away from all centers. The tuple of back-projected planes corresponding to the projection of $Y$ all contain the common line $\mathrm{span}\{X,Y\}$. By the fiber dimension theorem, we have $\dim \mathcal M_{\Ca}^{3,1}<3$. If there are three distinct line centers that do not meet in a common point, then three generic back-projected planes meet exactly in a point away from the three centers, implying that the dimension of the generic fiber is 0.
\end{proof}

\begin{proposition}\label{prop: dim L} For $N=3$ and $h\le 3$, let $\Ca$ be a camera arrangement. Assume $\gamma\in \{0,1,2,3\}$ and consider pairwise disjoint lines $L_j$ for $j=1,\ldots,\gamma$. Denote the set of lines meeting each $L_j$ by $\Lambda$. Then
\begin{align}
    \dim \mathcal L_{\Ca}^{\Lambda,3,h}=4-\gamma,
\end{align}
if and only if \textnormal{(1):} $h=3$, or \textnormal{(2):} $h=2$, $\gamma=0,1$ and there are two disjoint point centers, or \textnormal{(3):} $h=2$ and $\gamma=2,3$.
\end{proposition}

We recall that we always assume that the centers of $\Ca$ do not meet any $L_j$.

\begin{proof}$ $

\textnormal{(1):} Follows from \Cref{le: N N}.

\textnormal{(2):} The case $h=2$ and $\gamma=0$ was dealt with in \cite{breiding2023line}, and a general formula for non-anchored line multiview varieties is provided in \cite{rydell2024projections}. If $h=2$ and $\gamma=1$ and all centers are equal, then all back-projected planes are equal and the dimension is at most 2. If there are two cameras with disjoint centers, then take $H$ to be a generic plane that meets the first center; it does not meet the second. However, $H$ meets the anchoring line $L$ in a point. Take a generic line $L'$ through that point inside $H$. The back-projected planes of the images of $L'$ meet exactly at $L'$. The fiber dimension theorem says that the dimension is 3. 

\textnormal{(3):} It suffices to restrict to arrangements with one camera $C$ to calculate the dimension. First, assume $\gamma=2$. Any generic plane through $c$, the center of $C$, meets the two anchoring lines $L_1,L_2$ in unique points. There is then a unique line in $\Lambda$ contained in this back-projected plane. By genericity, it does not meet $c$. The dimension of the variety is therefore at least 2. Further, it cannot be more than 2, as it is the dimension of $\Lambda$. Second, assume $\gamma=3$. Assume by contradiction that the dimension of the multiview variety is $0$, i.e., the image of $\wedge^2 C$ restricted to $\Lambda$ consists of finitely many points. Then any line in $\Lambda$ would have to be contained in finitely many back-projected planes. This is not possible, since $\Lambda$ consists of all lines in one family of lines of a smooth quadric. The dimension is therefore 1, the dimension of $\Lambda$.  
\end{proof}


To get set-theoretic equations that describe point multiview varieties, we require more notation. Given $n$ cameras $C_1,\ldots,C_n$ of sizes $(h+1)\times (N+1)$, and $(h + 1)\times 1$ vectors $x_1,\ldots,x_n$, we define
\begin{align}
    M_\Ca(x):=\begin{bmatrix}
        C_1 & x_1 & 0 & \cdots & 0\\
        C_2 & 0 & x_2 & \cdots & 0\\
        \vdots & \vdots & \vdots & \ddots & \vdots \\
        C_n & 0 & 0 & \cdots & x_n
    \end{bmatrix}. 
\end{align}
This $(h+1)n\times (N+1+n)$ matrix appears in many places in the literature \cite{agarwal2019ideals,kileel2022snapshot}. Any point multiview variety is set-theoretically described by the condition that $M_{\Ca}(x)$ is rank-deficient, assuming a certain genericity of centers as detailed in \cite{li2018images}. In this work by Li, set-theoretic descriptions for point multiview varieties are further given for any camera arrangements. We specialize these results below to the point multiview varieties from \Cref{s: VV}.

Given some $N$ and $h$ and a camera arrangement $\Ca$, recall that $\mathcal M_{\Ca}^{N,h}$ is the point multiview variety we get as the closure of the image of the projection map $\Phi_{\Ca,0}^{N,h}$. 
Consider a camera $C:\PP^3\dashrightarrow\PP^1$. Given an image point $x\in \PP^1$, the back-projected plane $H$ is the (closure of the) set of points that are projected by $C$ onto $x$ and it equals
\begin{align}
    \{X\in \PP^3: (C^\top \tilde x)^\top X=0\}, \quad \textnormal{where} \quad \tilde x:= \begin{bmatrix}0 & 1 \\ -1 & 0 \end{bmatrix}x.
\end{align}
For centers $c_i$ and a set of indices $I$, denote the intersection of all $c_i$ with $i\in I$ by $c_I$.
        
\begin{theorem} In each case below, let $\Ca$ be an arrangement of $n$ cameras of appropriate sizes.

    \begin{enumerate}
       \item[(P1)] Let $N\ge 1$ be a natural number. Then
       \begin{align}
           \mathcal{M}_{\Ca}^{N,N}=\{x\in (\PP^N)^n \colon \rank \begin{bmatrix}C_1^{-1}x_1&\cdots & C_n^{-1}x_n\end{bmatrix}\le 1\}.
       \end{align} 
    \item[(P2)] Let $N\ge 2$ be a natural number. Then
    \begin{align}
        \mathcal{M}_{\Ca}^{N,N-1}=\{x\in (\PP^{N-1})^n \colon \rank M_{\Ca}(x) \le N+n  \}
    \end{align}
    if and only if the centers of $\Ca$ are pairwise distinct.
    \item[(P3)] Let $N\ge 1$ be a natural number. Then
    \begin{align}
        \mathcal{M}_{\Ca}^{N,1}=\{x\in (\PP^{1})^n \colon \rank [C_1^\top\widetilde{x}_1\; \cdots \; C_n^\top\widetilde{x}_n]\le N \}
    \end{align}
     if and only if for each $I\subseteq [n]$ with $c_I\neq \emptyset$, we have $N-|I|-1\ge \dim c_I$.
    \end{enumerate}
\end{theorem}
\begin{proof} $ $

\textit{(P1)}. The map $\Phi_{\Ca,0}^{N,N}$ is a morphism and its image equals $\mathcal M_{\Ca}^{N,N}$. Then $x\in (\PP^N)^n$ lies in $\mathcal M_{\Ca}^{N,N}$ if and only if there is an $X\in \PP^N$ such that $x_i=C_iX$ for each $i$. This condition is equivalent to $C_i^{-1}x_i\in \PP^N$ all being equal. This happens precisely when $\rank \begin{bmatrix}C_1^{-1}x_1&\cdots & C_n^{-1}x_n\end{bmatrix}=1$.  

\textit{(P2)}. Here, all camera centers are points. If two centers coincide, then the multiview variety satisfies the constraint that the corresponding back-projected lines corresponding to those two cameras are always the same, which is not captured by the rank-deficiency of $M_{\Ca}$. 

Now, assume that all centers are disjoint. Take $x\in (\PP^{N-1})^n$ such that $M_{\Ca}(x)$ is rank-deficient. Let $\underline{x}_i$ denote affine representatives of $x_i$ in $(\CC^{N})^n$. There is a vector $(\underline{X},-\lambda_1,-\ldots,-\lambda_n)$ in the kernel of $M_{\Ca}(\underline{x})$, i.e., $C_i\underline{X}=\lambda_i\underline{x}_i$ holds in $\CC^{N}$ for each $i$. Let $X\in \PP^N$ denote the projectivization of $\underline{X}$. We are done if $C_iX\neq 0$ for each $i$, because then $x_i=C_iX$ holds in $\PP^{N-1}$ for each $i$. If $C_iX=0$ for some $i$, then $C_jX\neq 0$ for $j\neq i$, since all centers are pairwise distinct. This means that $x_j=C_jX$ for each $j\neq i$. Let $L$ be the back-projected line of $x_i$. Choose any sequence $X^{(a)}\to X$ inside $L$. Then by construction, $x_i=C_iX^{(a)}$ for each $a$. Further $x_j=\lim C_jX^{(a)}$ for each $j\neq i$ and $x\in \mathcal M_{\Ca}^{N,N-1}$. For the other inclusion, note that by the above, $M_\Ca(x)$ is rank-deficient for any $x\in \mathrm{Im}\; \Phi_{\Ca,0}^{N,N-1}$. This condition must then also hold in the closure $\mathcal M_{\Ca}^{N,N-1}$.

\textit{(P3)}. By \cite[Section 4]{rydell2024projections}, $\mathcal M_{\Ca}^{N,1}$ is the set of image points $x=(x_1,\ldots,x_n)$ whose back-projected planes meet in a point if and only if the center arrangement of $\Ca$ is \textit{pseudo-disjoint}, which in turn is characterized by $N-|I|-1\ge \dim c_I$ for each $I\subseteq [n]$ with $c_I\neq \emptyset$. Since $C_i^\top\tilde x$ define the equations of the back-projected planes of $x$, the back-projected planes meet in at least a point precisely when $\begin{bmatrix}
           C_1^\top\widetilde{x}_1 & \cdots & C_n^\top\widetilde{x}_n 
        \end{bmatrix}$  is rank-deficient. 
\end{proof}

\begin{remark}\label{remark:vanishing-ideals}
For sufficiently generic camera arrangements, our set-theoretic descriptions of non-anchored point and line multiview varieties  correspond to stronger ideal-theoretic statements in the literature---see e.g.~\textnormal{\cite{agarwal2019ideals,agarwal2022atlas,LMI}}.
\end{remark}

Finally, we address the case of line multiview varieties and anchored versions thereof. We recall from \Cref{ss: MV}, that when mapping lines from $\PP^3$ to $\PP^2$, we send elements from $\Gr(1,\PP^3)$ to $(\PP^2)^*$, the dual of $\PP^2$. To be more precise, given a line spanned by $X_0,X_1$ in $\PP^3$ and a camera $C:\PP^3\to \PP^2$, the image of this line is $\ell=CX_0\times CX_1$. The set of points $\{x\in \PP^2: \ell^\top x=0\}$ coincides with $\{CX: X\in \mathrm{span}\{X_0,X_1\}\}$.  

\begin{theorem}\label{thm:line-anchored-equations} In each case below, let $\Ca$ be an arrangement of $n$ cameras of appropriate sizes.

    \begin{enumerate}
    \item[(L1)] Let $\gamma \in \{0,1,2,3\}$ and consider pairwise disjoint lines $L_j$ for $j=1,\ldots,\gamma $. Denote the set of lines meeting each $L_j$ by $\Lambda$. Then 
    \begin{align}\begin{aligned}
        \mathcal{L}_{\Ca}^{\Lambda,3,3}=\{\ell\in \mathrm{Gr}(1,\PP^3)^n& \colon \rank\begin{bmatrix}
            (\wedge^2 C_1)^{-1}\ell_1&\cdots & (\wedge^2 C_n)^{-1}\ell_n
        \end{bmatrix}\le 1,\\
        &\;\:\:\textnormal{ and } (\wedge^2 C_1)^{-1}\ell_1 \textnormal{ meets each }L_j, j=1,\ldots,\gamma \}.
    \end{aligned}
    \end{align}
    \item[(L2)] Let $\gamma\in \{0,1,2,3\}$ and consider pairwise disjoint lines $L_j$ for $j=1,\ldots,\gamma$. Denote by $\Lambda$ the variety of lines in $\PP^3$ that meet each $L_j$. Then
   \begin{equation}\label{eq: LCLi}
\mathcal{L}_{\Ca}^{\Lambda,3,2}
=
\left\{
\ell\in \mathrm{Gr}(1,\PP^2)^n \;\colon\;
\begin{aligned}
&\rank\begin{bmatrix}
C_1^\top\ell_1 & \cdots & C_n^\top\ell_n
\end{bmatrix}\le 2,\\
&\mathrm{Gr}(1,H_i)\cap \Lambda \neq \emptyset \textnormal{ for each }i,\\
&H_i\cap H_j\in\Lambda
\text{ for each pair }(i,j)
\end{aligned}
\right\}.
\end{equation}
    if and only if any line $E\in \Lambda$ contains at most $3-\gamma$ centers and, unless $\gamma=3$, all centers are pairwise distinct.
\end{enumerate}
\end{theorem}

By convention, if \(H_i=H_j\), we regard the condition \(H_i\cap H_j\in\Lambda\) as automatically satisfied. Algebraically, when \(H_i=H_j\), the Plücker coordinates of \(H_i\cap H_j\) is zero, and hence all homogeneous equations defining \(\Lambda\) vanish.

\begin{proof} $ $ 

\textit{(L1)}. The projection map is a morphism and $\mathrm{Im}\;\Phi_{\Ca,1}^{3,3}\restriction_{\Lambda}$ equals $\mathcal{L}_{\Ca}^{\Lambda,3,3}$. Moreover, $\wedge^2C_i$ are invertible maps, and $\Phi_{\Ca,1}^{3,3}$ is injective. Then $(\wedge^2 C_i)^{-1}$ takes the image lines $\ell=(\ell_1,\ldots,\ell_n)\in \mathcal{L}_{\Ca}^{\Lambda,3,3}$ to original line $L$ that projects onto each $\ell_i$. We get that $\ell \in \mathcal{L}_{\Ca}^{\Lambda,3,3}$ if and only if the lines $(\wedge^2 C_i)^{-1}\ell_i$ are projectively the same, say equal to $L$, and meet each of the lines $L_j$, meaning $L\in \Lambda$. 

\textit{(L2)}. $ $

$\Rightarrow)$ Assume first that \eqref{eq: LCLi} holds. Let $\gamma<3$. If two centers $c_1,c_2$ coincide, then consider any line $L$ through $c_1=c_2$ and each $L_j$. Any tuple $\ell=(\ell_1,\ldots,\ell_n)$ whose back-projected planes contain $L$ satisfies $\ell\in  \mathcal{L}_{\Ca}^{\Lambda,3,2}$. However, in the image of $\Phi_{\Ca,1}^{3,2}\restriction_{\Lambda}$, it is clear that $H_1=H_2$. This is a contradiction, and we conclude $c_1\neq c_2$.

Now assume that $E$ meets each $L_j$ and contains $4-\gamma$ or more centers, say $c_i$ for $i\in I$ for some set of indices $I\subseteq[n]$ with $|I|\ge 4-\gamma$. Then consider the variety 
   \begin{equation}
\mathcal B(E):=
\left\{
\ell\in \mathrm{Gr}(1,\PP^2)^n \;\colon\;
\begin{aligned}
&\ell \textnormal{ lies in the right-hand side of \eqref{eq: LCLi}},\\
&E\subseteq H_i \textnormal{ for each }i
\end{aligned}
\right\}.
\end{equation}
For each $i\in I$, there is a 1-dimensional degree of freedom in the choice of back-projected plane $H_i$ through $c_i$ and therefore of $\ell_i$ inside $\mathcal B(E)$. The variety $\mathcal B(E)$ is an irreducible and proper subvariety of $\mathcal{L}_{\Ca}^{\Lambda,3,2}$, at least of dimension $|I|$. We now arrive at a contradiction, because the multiview variety is irreducible and by the above and \Cref{prop: dim L}, $\dim \mathcal{L}_{\Ca}^{\Lambda,3,2}=4-\gamma\le |I|\le \dim \mathcal B(E)$.

$\Leftarrow)$ The inclusion $\subseteq$ holds in \eqref{eq: LCLi}, because the conditions of the right-hand side holds on the image of $\Phi_{\Ca,1}^{3,2}\restriction_{\Lambda}$. We show the other inclusion $\supseteq$. The conditions of \eqref{eq: LCLi} guarantee that it is enough to prove
\begin{align}\begin{aligned}
    \mathcal B(E)\subseteq \mathcal{L}_{\Ca}^{\Lambda,3,2} \text{ for any line } E\in \Lambda.
\end{aligned}
\end{align}
We consider two different cases: 1) $E$ meets no centers, 2) $E$ meets one or more centers. \phantom\qedhere

\begin{enumerate}[leftmargin =3.5em]
    \item[Case 1:] $\ell\in \mathcal B(E)$ is the image of $E$, showing that $\ell$ lies in $ \mathcal{L}_{\Ca}^{\Lambda,3,2}$.

    \item[Case 2:] Let $I\subseteq [n]$ be the indices of centers that meet $E$. By assumption, $3-\gamma \ge |I|$ and $|I|>0$. It follows that $\gamma <3$. Consider a generic tuple $\ell\in \mathcal B(E)$, and let $H_i$ be the corresponding back-projected planes. Fix some line $E^*$ disjoint from $E$ that meets each $L_j$. Consider the intersection lines $F_i:=H_i\cap \mathrm{span}\{c_i, E^*\}$ for $i\in I$. The number of $F_i$ and $L_j$ together is at most three. Each $F_i$ and $L_j$ meets $E$ and $E^*$. Then following \cite[Section 2]{breiding2023line}\cite[Section 5]{LMI}, there is a sequence of lines $L^{(a)}\to E$ in a smooth quadric that contains each $F_i$ and $L_j$ such that $L^{(a)}$ meets no centers and $H_i=\lim \; \mathrm{span}\{c_i,L^{(a)}\}$ for each $i\in [n]$. This shows that $\ell \in \mathcal{L}_{\Ca}^{\Lambda,3,2}$.\pushQED{\qed}  \qedhere
\popQED
\end{enumerate}
\end{proof}





\section{Multidegrees}\label{s: Multdeg}


Denote by~$L_{d}\subseteq \mathbb P^{h}$ a general linear subspace of codimension $d$, meaning dimension $h-d$. The \textit{multidegree} of a variety $\mathcal{X}\subseteq \PP^{h_1}\times \cdots \times\PP^{h_n}$ is the function 
\begin{align}
    D(d_{1},\dots,d_{n}):= \#( \mathcal{X}\cap (L_{d_{1}}^{(1)}\times \cdots\times L_{d_{n}}^{(n)})), 
\end{align}
for $(d_{1},\dots,d_{n})\in \mathbb N^{n}$ such that $d_{1}+\cdots +d_{n} = \operatorname{dim} \mathcal{X}$ and $d_i\le h_i$. The multidegree is a natural property to study; it measures how non-linear a variety is. From the point of view of computer vision, it tells us what information we need in order to uniquely reconstruct world objects. For instance, given an arrangement of three generic cameras $\Ca$, $\mathcal L_\Ca^{3,2}$ has $D(2,1,1,0,\ldots,0)=1$. This means that given an image line $\ell$ in one image plane, and two points $x',x''$ in two other image planes, there is a unique reconstruction of a world line $L$ that projects onto $\ell$ in the first image plane and whose projections onto the other image planes contain the points $x'$ and $x''$, respectively. 

For (anchored) multiview varieties with generic centers, the function $D$ is symmetric. This implies that for any permutation $\sigma \in S_n$, $D(d_{1},\dots,d_{n})$ is equal to $D(d_{\sigma{(1)}},\dots,d_{\sigma{(n)}})$. If the multidegree of a variety is constant and equals $d$, we write $D=d$. 

\begin{proposition}\label{prop: multideg} The values of the multidegree function for the 14 (anchored) multiview varieties appearing in \Cref{thm: ED-cat}, for generic camera arrangements whose projection maps are generically injective, are as follows.
\begin{enumerate}
    \item $\mathcal{M}^{N,h}$ \textnormal{:} $D=1$, 
    \item $\mathcal{L}^{\gamma \mathrm{line},3,3}$ \textnormal{:} $D=2$,
    \item $\mathcal{L}^{3,2}$ \textnormal{:} $D(2,2,\ldots,0)=1,D(2,1,1,0,\ldots,0)=1,D(1,1,1,1,0,\ldots,0)=2$,
    \item $\mathcal L^{\mathrm{line},3,2}$ \textnormal{:} $D(2,1,\ldots,0)=1,D(1,1,1,0,\ldots,0)=2$,
     \item $\mathcal L^{2\mathrm{line},3,2}$ \textnormal{:} $D(2,0,\ldots,0)=1,D(1,1,0,\ldots,0)=2$,
      \item $\mathcal L^{3\mathrm{line},3,2}$ \textnormal{:} $D(1,0,\ldots,0)=2$.
    \end{enumerate}
\end{proposition}

For point multiview varieties, a general formula for the multidegrees is found in \cite{li2018images}. 

\begin{proof} The argument provided for the multidegrees of $\mathcal L^{3,2}$ in \cite[Section 4]{breiding2023line} extends directly for both the point multiview varieties and all (anchored) line multiview varieties with $h=2$. We are left to prove the second item.

In the case of line multiview varieties with $h=3$, note that each such variety lies in $(\PP^5)^n$, where $n$ is the number of cameras. Let $\Lambda$ be the set of lines meeting each $L_j,j=1,\ldots,\gamma$. Taking a generic hyperplane of $\PP^5$ in image $i$ corresponds to a generic hyperplane in the domain via the inverse mapping $(\wedge^2C_i)^{-1}$. Then for the calculation of multidegrees, we wish to determine the cardinality of $\Lambda$ intersected with $4-\gamma$ generic hyperplanes of $\PP^5$. By definition, each multidegree is then the degree of $\Lambda$. 

Now, $\Lambda$ is cut out by one degree-2 equation, namely the defining equation of $\Gr(1,\PP^3)$, and $\gamma$ linear equations. Therefore it is at most of degree 2. More precisely, the set of lines $\Lambda(L)$ meeting a fixed line $L$ is the intersection of $\Gr(1,\PP^3)$ with a hyperplane $H_L$. Let $L_{\gamma+1},\ldots,L_4$ be generic in $\PP^3$. Then 
\begin{align}\begin{aligned}
    &\deg \Big(\Lambda(L_1)\cap \Lambda(L_2)\cap \Lambda(L_3)\cap \Lambda(L_4)\Big)\\
    =&\deg \Big(\Gr(1,\PP^3)\cap H_{L_1}\cap H_{L_2}\cap H_{L_3}\cap H_{L_4}\Big)\\
    =&\, 2.
\end{aligned}
\end{align}
Replacing some of $H_{L_i}$ by generic hyperplanes in $\PP^5$, we get the same degree. This proves that $\deg(\Lambda)=2$ for each $\gamma$. 
\end{proof}



\section{Resectioning Varieties}\label{s:resectioning}

Recall from the introduction that the study of multiview varieties is motivated by a classical problem of metric algebraic geometry known as \textit{triangulation}---given $n$ cameras $\PP^N \dashrightarrow \PP^h,$ recover the scene point in $\PP^N$ of best fit to $n$ given projections in $\PP^h.$
In this section, we consider the ``dual problem" of \textit{resectioning}---given $n$ scene points in $\PP^N$, recover the camera of best fit to $n$ given projections in $\PP^h.$
Variants of this problem play a key role in applications such as visual localization~\cite{sattler2018benchmarking}.
The natural analogues of point multiview varieties for the resectioning problem have been studied in two recent works in the special case $(N,h)=(3,2)$~\cite{agarwal2022atlas,connelly2025algebra}.
Here we summarize the general situation.

Fix integers $h \le N$ and $n,$ and a configuration of points $\mathbf{X} = (X_1, \ldots , X_n) \in \left( \PP^N \right)^n$.
We assume the points $\mathbf{X}$ are in \textit{linear general position}---that is, $n>N$ and no subset of $(N+1)$ points $X_{i_0}, \ldots , X_{i_N}$ lie in a common hyperplane.
We consider an analogue of the map~\eqref{eq: jointmap}:
\begin{align}\label{eq:jointmap-resectioning}\begin{aligned}
\Psi_{\mathbf{X}}^{N,h} : \PP (\CC^{(h+1)\times (N+1)}) &\dashrightarrow (\PP^h)^n, \\
C &\mapsto (C X_1, \ldots , C X_n).    
\end{aligned}
\end{align}
\begin{definition}\label{def:resectioning-variety}
The \textit{resectioning variety} $\mathcal{R}_{\mathbf{X}}^{N,h}$ associated to the point arrangement $\mathbf{X}$ is the closed image of the rational map~\eqref{eq:jointmap-resectioning}.
\end{definition}
\begin{remark}
The rational map~\eqref{eq:jointmap-resectioning} makes sense for an arbitrary point arrangement $\mathbf{X}.$
Our assumption of linear general position on $\mathbf{X}$ ensures that~\Cref{def:resectioning-variety} is equivalent to the definition of $\mathcal{R}_{\mathbf{X}}^{N,h}$ used in previous papers~\cite[Proposition 7]{connelly2025algebra}. It also implies that $\mathcal{R}_{\mathbf{X}}^{N,h}$ is a proper subvariety of $(\PP^h)^n $ for\phantom\qedhere
\[\pushQED{\qed} 
n \ge 
\left\lfloor \displaystyle\frac{Nh + N + h}{h} \right\rfloor
+1
=
N+2 + \left\lfloor \displaystyle\frac{N}{h} \right\rfloor.\qedhere
\popQED
\] 
\end{remark}

Resectioning varieties form a special class of higher-dimensional point multiview varieties.
To see this, we may associate to each point $X_i$ of $\mathbf{X}$ the \textit{dual camera}
\begin{equation}\label{eq:dual-camera}
C_{X_i} := 
I_{h+1} \otimes X_i^\top 
\in
\PP(\CC^{(h+1) \times (Nh + N + h+1)})
,
\end{equation}
where $I_{\bullet }$ denotes an identity matrix and $\otimes $ is the usual Kronecker product of matrices.
For the associated camera arrangement, write $\Ca_\mathbf{X} = (C_{X_1}, \ldots, C_{X_n}).$
A camera matrix $C \in \PP(\CC^{(h+1)\times (N+1)})$ may be vectorized by writing its entries in row-major order. Let $X_{C}$ denote the result of this vectorization operation.
From the identity $C X_i = C_{X_i} X_{C}$, we have
\begin{equation}\label{eq:resect-is-multiview}
\mathcal{R}_{\mathbf{X}}^{N,h} = \mathcal{M}_{\Ca_\mathbf{X}}^{Nh+N+h,h}.
\end{equation}

\begin{example}\label{ex:segre}
When $(N,h)=(1,1)$, we obtain the dual camera $C_X$ from a world point $X$,
\begin{align}
    X = \begin{bmatrix}
U_1\\
V_1
\end{bmatrix}
\in \PP^1 \quad \Rightarrow \quad 
C_X = \begin{bmatrix}U_1 & V_1 & 0 & 0 \\
0 & 0 & U_1 & V_1 \end{bmatrix} \in \PP (\CC^{2\times 4}).
\end{align}
For $\mathbf{X}=(X_1,X_2,X_3 ,X_4)$ in linear general position, the resectioning variety $\mathcal{R}_{\mathbf{X}}^{1,1} \subseteq (\PP^1)^4$
is the hypersurface defined by the quadrilinear form
\begin{align}
\det 
\begin{bmatrix}
C_{X_1} & x_1 & & & \\
C_{X_2} &  & x_2 & & \\
C_{X_3} &  & & x_3 & \\
C_{X_4} &  & & & x_4 
\end{bmatrix} = 0.
\end{align}
To obtain a dual world point $X_C$ from a camera $C$, we proceed similarly:
\begin{align}
C = \begin{bmatrix}
    c_{1,1} & c_{1,2}\\
    c_{2,1} & c_{2,2}
\end{bmatrix}
\quad \Rightarrow \quad 
X_C = \begin{bmatrix}
c_{1,1} &
c_{1,2} &
c_{2,1} &
c_{2,2}
\end{bmatrix}^\top.
\end{align}
Noting for any $X\in \PP^1$ as above that
\begin{align}
\PP (\ker (C_X)) = \{ s \begin{bmatrix}
V_1 & -U_1 & 0 & 0
\end{bmatrix}^\top + 
t \begin{bmatrix}
0 & 0 & V_1 & -U_1
\end{bmatrix}^\top
\colon [s:t] \in \PP^1 \}, 
\end{align}
we see that all centers of any arbitrary arrangement of the form $\mathbf{C}_{\mathbf{X}}$ must lie on the set
\begin{align}
\displaystyle\bigcup_{X \in \PP^1} \{ C \in \PP (\CC^{2\times 2}) \colon X_C \in \ker (C_X) \} 
=
\{ C \in \PP (\mathbb{C}^{2\times 2}) \colon \rank C = 1 \},
\end{align}
the smooth Segre quadric in $\PP^3.$
This is not the case for generic $\PP^3 \to \PP^1$ cameras.
\end{example}

The resectioning problem is, of course, very interesting to study in the more general setting of line or anchored multiview varieties.
However, it is not evident whether or not we can realize more general resectioning varieties as multiview varieties in the same vein as~\eqref{eq:resect-is-multiview}.
For this reason, we focus solely on the case of points, but suggest the study of line or anchored resectioning varieties as directions worthy of future research.

Adapting the program of the previous sections, it is perhaps little surprise that there are no ED-equivalences between the (point) resectioning varieties, and that their multidegrees and set-theoretic equations are inherited from the multiview varieties.

We also point out the following ideal-theoretic result, which was originally stated for the special case $(N,h) = (3,2).$ 
The proof applies verbatim to the general case.

\begin{proposition}
[{\hspace{1sp}\cite[Theorem 6]{connelly2025algebra}}]\label{thm:resectioning-vanishing-ideal}
For $\mathbf{X}$ in linear general position, the vanishing ideal of $\mathcal{R}_{\mathbf{X}}^{N,h}$ is generated by all $k$-linear forms obtained as maximal minors of the matrices  
\begin{align}
\begin{bmatrix}
C_{X_{i_1}} & x_{i_1} & \\
\vdots & & \ddots \\
C_{X_{i_k}} & & &  x_{i_k}
\end{bmatrix},
\end{align}
where $N+2 + \left\lfloor \frac{N}{h} \right\rfloor  \le k \le (N+1)(h+1)$ and $\{ i_1, \ldots , i_k \}$ ranges over all $k$-element subsets of $[n]$.
Moreover, these generators form a universal Gr\"{o}bner basis.
\end{proposition}

Just as for multiview varieties, we write $\mathcal{R}_n^{N,h}$ in place of $\mathcal{R}_{\mathbf{X}}^{N,h}$ when it is understood that the point arrangement is sufficiently generic.
In view of~\eqref{eq:resect-is-multiview} and our census of ED degrees for point multiview varieties, it might seem reasonable to conjecture that $\mathrm{EDD} (\mathcal{R}^{N,h}_n)$ is a polynomial of degree $Nh + N+ h$ in $n.$
However, computational experiments suggest a polynomial of much lower degree.
\begin{conjecture}\label{conj:lower-degree}
$\mathrm{EDD} (\mathcal{R}^{N,h}_n\:\cap \:(\RR^h)^n)$ is a degree-$N$ polynomial in $n$ for $n\ge N+2 + \left\lfloor \frac{N}{h} \right\rfloor.$
\end{conjecture}

To understand the discrepancy between $\mathrm{EDD} (\mathcal{R}_n^{N,h} )$ and $\mathrm{EDD} (\mathcal{M}_n^{Nh+N+h,h} )$, we recall that the point multiview variety $\mathcal{M}_{\Ca}^{N,h}$ is smooth for any sufficiently generic arrangement $\Ca$ of $n \gg 0$ cameras.
However, the camera arrangements $\Ca_\mathbf{X}$ \textit{are not generic}, as we will now observe.
For fixed $N$ and $h$, $1\le r \le h+1,$ let us define the constant-rank sets
\begin{align}\begin{aligned}
\mathcal{V}_{n,r}^{h} &= \{ (x_1, \ldots , x_n) \in \left( \PP^h \right)^n \colon \rank [x_1\; \cdots \; x_n] = r \},\\
\mathcal{W}_r^{N,h} &= \{ A \in \PP \left( \CC^{(h+1)\times (N+1)} \right) \colon \rank A = r \}.\end{aligned}
\end{align}
For example, $\mathcal{V}_{n,1}^h$ is the image of $\PP^h$ under its $n$-fold diagonal embedding, and
$\mathcal{W}_1^{N,h}$ is the image of $\PP^N \times \PP^h$ under the Segre embedding.
\Cref{thm:sing-res11} demonstrates the fact that resectioning varieties, unlike multiview varieties, are generally \textit{singular}.

A formidable theory of Euclidean distance degrees has been developed for singular varieties---see e.g.~\cite{ed-singular-projective}.
Determining singular loci, and more generally Whitney stratifications, plays an important role in carrying out the calculations of this theory.
For the singularities of resectioning varieties, our next result addresses the simplest case.

\begin{theorem}\label{thm:sing-res11}
For $n\ge 4,$ and $\mathbf{X} \in (\PP^1)^n$ in linear general position, 
\begin{equation}
\left(\mathcal{R}_{\mathbf{X}}^{1,1} \right)_{\mathrm{sing}} = \mathcal{V}_{n,1}^1.
\end{equation}
\end{theorem}
For general $\mathcal{R}_{\mathbf{X}}^{N,h},$ the singularities appear to be more complicated. 
Using Macaulay2~\cite{grayson2002macaulay2}, we find that $\left(\mathcal{R}_n^{2,1}\right)_{\mathrm{sing}}$ has several irreducible components. 
Also,
$\mathcal{V}_{5,2}^2 \not\subset \left(\mathcal{R}_5^{2,2}\right)_{\mathrm{sing}}$, whereas $\mathcal{V}_{n,2}^2 \not\subset \mathcal{R}_n^{2,2}$ for $n>5.$
To prove~\Cref{thm:sing-res11}, we record the following structural result about resectioning varieties whose points consist of homographies $C: \PP^N \to \PP^N.$

\begin{proposition}\label{prop:iso-resction} For an arrangement $\mathbf{X} \subset \left(\PP^N\right)^n$ of $n\ge N+2$ points in linear general position, the rational map $\Psi_{\mathbf{X}}^{N,N}$ restricts to an isomorphism
\begin{equation}\label{eq:psi-resction-restricted}
\mathcal{W}_{N+1}^{N,N} \xrightarrow{\sim} \operatorname{im} \Psi_{\mathbf{X}}^{N,N}
\cap \mathcal{V}_{n,N+1}^N.
\end{equation}
\end{proposition}
In particular, for $n=N+2$ general points this gives $\mathcal{W}_{N+1}^{N, N} \cong \operatorname{PGL}_{N+1}.$
\begin{proof}
Without loss of generality, we may assume $X_1 = E_1 , \ldots , X_{N+2} = E_{N+2}$ form the standard projective basis. Given $C \in \mathcal{W}_{N+1}^{N,N},$ choose a representative in homogeneous coordinates whose columns are $c_1, \ldots , c_{N+1}.$
Then
\begin{align}
\Psi_{\mathbf{X}} (C) \sim (C E_1, \ldots , C E_{N+1}, \ldots ) = (c_1, \ldots , c_{N+1}, \ldots ) \in \mathcal{V}_{n,N+1}^N.
\end{align}
This shows the map~\eqref{eq:jointmap-resectioning} is defined on $\mathcal{W}_{N+1}^{N,N}$, and its image is contained in $\operatorname{im} \Psi_{\mathbf{X}}^{N,N} \cap \mathcal{V}_{n,N+1}^N$.
To prove the reverse inclusion, let $(x_1, \ldots , x_n) \in \operatorname{im} \Psi_{\mathbf{X}}^{N,N} \cap \mathcal{V}_{n,N+1}^N$ and consider any $C\in \Psi_{\mathbf{X}}^{-1} (x_1, \ldots , x_n).$
From the first $(N+1)$ components $x_1, \ldots , x_{N+1},$ we must have
\begin{align}
  C \sim [\lambda_1 x_1 \cdots \lambda_{N+1} x_{N+1}]  
\end{align}
for some scalars $\lambda_1, \ldots , \lambda_{N+1}.$
Moreover, there exists a scalar $\lambda_{N+2}$ such that 
\begin{align}
    \lambda_{N+2} x_{N+2} = C E_{N+2} = C (E_1 + \cdots + E_{N+1}) = [x_1 \cdots x_{N+1}] \begin{bmatrix}
\lambda_1 \\
\vdots \\
\lambda_{N+1}
\end{bmatrix},
\end{align}
which implies
\begin{align}
\begin{bmatrix}
\lambda_1 \\
\vdots \\
\lambda_{N+1}
\end{bmatrix}
\sim 
[x_1 \cdots x_{N+1}]^{-1} x_{N+2}
\end{align}
This shows that the restricted map~\eqref{eq:psi-resction-restricted} has a regular inverse on $\operatorname{im} \mathcal{\psi}_{\mathbf{X}}^{N,N} \cap \mathcal{V}_{n,N+1}^N$,\phantom\qedhere
\[\pushQED{\qed} 
(x_1, \ldots , x_n) \mapsto [x_1 \cdots x_{N+1}] \cdot \operatorname{diag} \left( 
[x_1 \cdots x_{N+1} ]^{-1} x_{N+2}
\right).\qedhere
\popQED
\]
\end{proof}

\begin{proof}[Proof of~\Cref{thm:sing-res11}]
Let us write $x_i = [u_i \phantom{f} v_i]^\top$ and $X_i= [a_i \phantom{f} b_i]^\top$ for $i=1,\ldots , n$. We begin by showing both the inclusion $\mathcal{V}_{n,1}^1 \subset \mathcal{R}_{\mathbf{X}}^{1,1}$ and that all points of $\mathcal{V}_{n,1}^1 $ are singular on $\mathcal{R}_{\mathbf{X}}^{1,1}.$
We establish an even stronger result by direct calculation using~\Cref{thm:resectioning-vanishing-ideal}; the quadrilinear generators of the vanishing ideal and all of their partial derivatives with respect to all $u_i$ and $v_i$ vanish on all points of $\mathcal{V}_{n,1}^1.$ To see this, we apply row permutations, Schur complements, and Laplace expansion to write each of these generators as a sum of products of certain pairs $2\times 2$ determinants, e.g.
\begin{align}
\det 
\begin{bmatrix}
X_1^\top & & u_1 \\
& X_1^\top & v_1\\
X_2^\top & & & u_2 \\
& X_2^\top & & v_2\\
X_3^\top & & & & u_3 \\
& X_3^\top & & & v_3\\
X_4^\top & & & & & u_4 \\
& X_4^\top & & & & v_4
\end{bmatrix}
&= \pm \det 
\begin{bmatrix}
X_1^\top & & u_1 \\
X_2^\top & & & u_2 \\
X_3^\top & & & & u_3 \\
X_4^\top & & & & & u_4 \\
& X_1^\top & v_1\\
& X_2^\top & & v_2\\
& X_3^\top & & & v_3\\
& X_4^\top & & & & v_4
\end{bmatrix} 
\\
&= 
\pm 
\det
\begin{bmatrix}
v_1 X_1^\top & -u_1 X_1^\top\\
v_2 X_2^\top & -u_2 X_2^\top\\
v_3 X_3^\top & - u_3 X_3^\top\\
v_4 X_4^\top & - u_4 X_4^\top\\
\end{bmatrix} 
\label{eq:schur-expanded-resectioning-constraint}
\\
&= 
\displaystyle\sum_{\{ i, j \} \sqcup \{ k, l \} = [4]}
\pm 
\det 
\begin{bmatrix}
a_i u_i & a_i v_i \\
a_j u_j & a_j v_j
\end{bmatrix}
\cdot 
\det 
\begin{bmatrix}
b_k u_k & b_k v_k \\
b_l u_l & b_l v_l
\end{bmatrix}.
\label{eq:laplace-expanded-resectioning-constraint}
\end{align}
Each of the $2\times 2$ determinants appearing in a summand of~\eqref{eq:laplace-expanded-resectioning-constraint} is easily seen to vanish when evaluated at a rank-1 matrix $\begin{bmatrix}
u_1 & \cdots & u_n\\
v_1 & \cdots & v_n
\end{bmatrix} \in \mathcal{V}_{n,1}^1.
$
Moreover, the partial derivatives of each summand also vanish on $\mathcal{V}_{n,1}^1.$
Thus $\left(\mathcal{R}_{\mathbf{X}}^{1,1} \right)_{\mathrm{sing}} \supseteq \mathcal{V}_{n,1}^1.$

To prove the reverse inclusion, let us first identify the boundary set $\mathcal{R}_{\mathbf{X}}^{1,1} \setminus \operatorname{im} \Psi_{\mathbf{X}}^{1,1}.$
A sequence of points $(\psi_{\mathbf{X}}^{1,1}(A_k))_{k\ge 1}$ in the image that converges to a point in the boundary $(x_1, \ldots , x_n) \in \mathcal{R}_{\mathbf{X}}^{1,1} \setminus \operatorname{im} \psi_{\mathbf{X}}^{1,1}$ must have the property that a limit point of sequence $(A_k)_{k\ge 1}$ in $\mathbb{P}^3$ is represented by some rank-1 matrix $A.$
Since points in $\mathbf{X}$ are distinct, $X_i \in \ker A $ for at most one $i$; moreover, such an $i$ must exist since otherwise $\Psi_{\mathbf{X}}(A)$ would be defined.
Up to relabeling, we may assume that $i=1$ and hence $x_2=x_3=\ldots = x_n.$

Now, using~\Cref{prop:iso-resction}, any point in the open subset $\operatorname{im} \Psi_{\mathbf{X}}^{1,1} \cap \mathcal{V}_{n,2}^1\subset \mathcal{R}_{\mathbf{X}}^{1,1}$ is smooth.
The same is in fact true for any $(x_1, \ldots , x_n) \in \mathcal{R}_{\mathbf{X}}^{1,1} \cap \mathcal{V}_{n,2}^1$.
To see this, note by the preceding paragraph that we may assume $x_1 = e_1+e_2,$ $x_2 = x_3 = \cdots = x_n = e_2$ and work in the affine chart where $v_i=1$ for all $i$, so that the point $(x_1, \ldots , x_n)$ is given in local coordinates by 
$$
u_* = 
(u_1,u_2,\ldots, u_n) = (1,0,\ldots, 0).
$$
Equation~\eqref{eq:schur-expanded-resectioning-constraint} for the index set $\{ 1,2,3,k \}$, where $k\ge 4,$ gives us the equation
$$
F_k(u) = 
\det 
\begin{bmatrix}
a_1 & b_1 & -u_1 a_1 & -u_1b_1\\
a_2 & b_2 & -u_2 a_2 & -u_2b_2\\
a_3 & b_3 & -u_3 a_3 & -u_3b_3\\
a_k & b_k & -u_k a_k & -u_kb_k
\end{bmatrix}=0.
$$
Differentiating, we have
$$
\displaystyle\frac{\partial F_k}{\partial u_k} (u_*) =
(a_2b_3-a_3b_2) (a_1 b_k - a_k b_1) \ne 0,
$$
since the points in $\mathbf{X}$ are assumed distinct.
Moreover, for $j> k,$ we have
$$
\displaystyle\frac{\partial F_k}{\partial u_j} (u_*) =0,
$$
and thus the Jacobian submatrix $\left(\displaystyle\frac{\partial F_k}{\partial u_j} (u_*)\right)_{4 \le k,j \le n}$ is invertible.
Since the resectioning variety is $3$-dimensional, we may conclude that $(x_1, \ldots , x_n)$ is a smooth point.
\end{proof}

\section{EDD Conjectures}\label{s: conj}

In this section, we state theorems and conjectures on Euclidean distance degrees for all (anchored) multiview varieties introduced in this paper, up to the ED-equivalence described in \Cref{thm: ED-cat},
as well as for several resectioning varieties discussed in~\Cref{s:resectioning}.
Our conjectures are supported by homotopy continuation computations in \texttt{Julia} \cite{bezanson2012julia,breiding2018homotopycontinuation} and \texttt{Macaulay2}~\cite{grayson2002macaulay2,Duff-Monodromy}. The code used for the numerical ED-degree computations is included
as ancillary files with the arXiv version of this paper. We emphasize that in this section, all camera arrangements are generic. 

\begin{remark}\label{rem:new-results}
Since the original version of this article appeared, 
results resolving particular cases have been announced in
works of Finkel and Rodriguez~\cite{finkel2025euclidean} and Hrub\'{y} and Shehu~\cite{hruby2026triangulation}.
We have marked these particular cases below with asterisks ($\ast $).    
\end{remark}

We consider generic arrangements of cameras and dual cameras, in the case of resectioning. Given a family of (anchored) multiview varieties $\mathcal M_{k}^{\Omega,N,h}$, we write $\mathcal M_{n,k}^{\Omega,N,h}$ or $\mathcal M_n^{\Omega,N,h},\mathcal L_n^{\Omega,N,h},$ $\mathcal P_n^{\Omega,N,h},$ for $k=0,1,2,$ respectively, to denote a generic (anchored) multiview variety $\mathcal M_{\Ca,k}^{\Lambda,N,h}$ given a camera arrangement $\Ca$ of $n$ cameras. By generic, we mean that $\Lambda$ is generic in $\Omega$ and that $\Ca$ consists of generic cameras. Similarly, $\mathcal{R}^{N,h}_n$ denotes a resectioning variety $\mathcal R_{\mathbf{X}}^{N,h}$ for $n$ generic configuration points.

In order to compute the EDD of a variety in multiprojective space, we first intersect it with an affine patch (a product of affine patches in each projective space). Throughout this section, we work with the standard affine patch $x_0=1$ in each projective space. For an (anchored) multiview variety $\mathcal M_{\Ca,k}^{\Lambda,N,h}$, we write $ \mathcal M_{\Ca,k}^{\Lambda,N,h}\:\cap\: (\RR^\eta)^n$ to denote the intersection with this standard affine patch. Recall that $\eta={h+1\choose k+1}-1$ is the dimension of the projective space that $\Gr(k,\PP^h)$ is embedded in via the Plücker embedding. 


\subsection{Point multiview varieties}\label{ss: conj point} The following EDDs are known for point multiview varieties.

\begin{theorem}[{\hspace{1sp}\cite[Section 4]{EDDegree_point}\cite[Theorem 1.7]{Rydell_2023_ICCV}}] $ $
\begin{enumerate}
    \item For $n\ge 3$, $\mathrm{EDD}(\mathcal{M}_n^{3,2}\:\cap\: (\RR^2)^n)=\displaystyle\frac{9}{2}n^3-\displaystyle\frac{21}{2}n^2+8n-4$,
    \item For $n\ge 3$, $\mathrm{EDD}(\mathcal{M}_n^{2,1}\:\cap\: (\RR^1)^n)=\displaystyle\frac{9}{2}n^2-\displaystyle\frac{19}{2}n+3,$
    \item For $n\ge 1$, $\mathrm{EDD}(\mathcal{M}_n^{1,1}\:\cap\: (\RR^1)^n)=3n-2.$
\end{enumerate}
\end{theorem}

The remaining conjectures are suggested by numerical homotopy continuation computations. 

\begin{conjecture}
$ $
\begin{enumerate}
    \item For $n\ge 1$, $\mathrm{EDD}(\mathcal{M}_n^{3,3}\:\cap\: (\RR^3)^n)=\displaystyle\frac{9}{2}n^3-\displaystyle\frac{21}{2}n^2+11n-4$,
  \item For $n\ge 4$, $\mathrm{EDD}(\mathcal{M}_n^{3,1}\:\cap\: (\RR^1)^n)=\displaystyle\frac{9}{2}n^3-\displaystyle\frac{39}{2}n^2+22n-4,$
  \item ($\ast $,~\cite{hruby2026triangulation}) For $n\ge 1$, $\mathrm{EDD}(\mathcal{M}_n^{2,2}\:\cap\: (\RR^2)^n)=\displaystyle\frac{9}{2}n^2-\displaystyle\frac{13}{2}n+3.$
\end{enumerate}
\end{conjecture}


\subsection{Line multiview varieties}\label{ss: conj line} Numerical computations in \texttt{HomotopyContinuation}~\cite{breiding2018homotopycontinuation} provide the following conjectures.

\begin{conjecture}
$ $
\begin{enumerate}
\item For $n\ge 2$, $  \mathrm{EDD}(\mathcal{L}_n^{3,3}\:\cap\: (\RR^5)^n)=\displaystyle\frac{27}{4}n^4-\displaystyle\frac{27}{2} n^3+\frac{109}{4} n^2-\displaystyle\frac{43}{2} n+6,$
    \item For $n\ge 4$, $ \mathrm{EDD}(\mathcal{L}_n^{3,2}\:\cap\: (\RR^2)^n)=\displaystyle\frac{27}{4}n^4-27n^3+\displaystyle\frac{121}{4}n^2-13n+6.$
\end{enumerate}
\end{conjecture}

\begin{conjecture} $ $
\begin{enumerate}
    \item For $n\ge 1$, $\mathrm{EDD}(\mathcal{L}_n^{\mathrm{line},3,3}\:\cap\: (\RR^5)^n)=9n^3-12n^2+15n-6,$
    \item For $n\ge 4$, $  \mathrm{EDD}(\mathcal{L}_n^{\mathrm{line},3,2}\:\cap\: (\RR^2)^n)=9n^3-21n^2+14n-6,$
     \item For $n\ge 1$, $\mathrm{EDD}(\mathcal{L}_n^{2\mathrm{line},3,3}\:\cap\: (\RR^5)^n)=9n^2-7n+4$,
     \item For $n\ge 2$, $\mathrm{EDD}(\mathcal{L}_n^{2\mathrm{line},3,2}\:\cap\: (\RR^2)^n)=9n^2-10n+4,$
         \item ($\ast $,~\cite{finkel2025euclidean}) For $n\ge 1$, $\mathrm{EDD}(\mathcal{L}_n^{3\mathrm{line},3,3}\:\cap\: (\RR^5)^n)=6n-2,$
         \item ($\ast $,~\cite{finkel2025euclidean}) For $n\ge 1$, $\mathrm{EDD}(\mathcal{L}_n^{3\mathrm{line},3,2}\:\cap\: (\RR^2)^n)=6n-2.$
\end{enumerate}
\end{conjecture}

We observe two patterns among these conjectures. Except for the two families anchored at lines meeting one fixed line,
the constant terms in these formulas equal the signed Euler
characteristics $(-1)^{\dim \Lambda}\chi(\Lambda)$ of their domains.
In the exceptional cases, the constant term is $-6$, whereas the
signed Euler characteristic is $-5$. Further, the top coefficients are equal to 
\begin{align}
    \frac{3^{\dim \Lambda}}{\dim \Lambda!}d_{\mathds{1}},
\end{align}where $d_{\mathds{1}}$ is the symmetric $D(1,\ldots,1,0,\ldots,0)$ multidegree of the multiview variety from \Cref{s: Multdeg}. We leave it for future work to investigate these patterns in more detail.

The lower bound on $n$ assumed in the theorem and conjectures corresponds to the number of generic cameras needed for the blowup to be isomorphic to the (anchored) multiview variety. For non-anchored multiview varieties, this question is studied in \cite[Section 6]{rydell2024projections}. 

\subsection{Resectioning}\label{ss:resect-EDD}
We conclude this section with conjectural Euclidean distance degrees for the resectioning varieties of~\Cref{s:resectioning}.
All formulas are new to the best of our knowledge,  except the one for $\mathcal{R}_n^{3,2}$ which is~\cite[Conjecture 23]{connelly2025algebra}.
 \begin{conjecture}\label{conj:resect-ED} $ $
 \begin{enumerate}
     \item For $n\ge 4$, $ \mathrm{EDD}(\mathcal{R}_n^{1,1}\:\cap \:(\RR^1)^n)= 3n - 8$,
     \item For $n\ge 5$, $ \mathrm{EDD}(\mathcal{R}_n^{2,2}\:\cap \:(\RR^2)^n)= 12n^2 -84n + 147$,
     \item For $n\ge 6$, $ \mathrm{EDD}(\mathcal{R}_n^{2,1}\:\cap \:(\RR^1)^n)= \displaystyle\frac{15}{2} n^2 - \displaystyle\frac{115}{2} n + 108$,
     \item For $n\ge 6,$ $\mathrm{EDD}(\mathcal{R}_n^{3,3}\:\cap \:(\RR^3)^n) = \displaystyle\frac{88}{3}n^{3}-400\,n^{2}+\displaystyle\frac{5456}{3}n-2756$,
          \item For $n\ge 6,$ $\mathrm{EDD}(\mathcal{R}_n^{3,2}\:\cap \:(\RR^2)^n) = \displaystyle\frac{80}{3}n^{3}-368 n^{2}+ \displaystyle\frac{5068}{3} n- 2580$,
     \item For $n\ge 8,$ $\mathrm{EDD}(\mathcal{R}_n^{3,1}\:\cap \:(\RR^1)^n) = \displaystyle\frac{21}{2}n^{3}-\displaystyle\frac{317}{2}n^{2}+793 n-1312$.
     \end{enumerate}
 \end{conjecture}

In closing, we note the implications of lower-than-expected degrees in~\Cref{conj:resect-ED} for global approaches to optimization over resectioning varieties based on polynomial system solving. 
For instance, $\mathrm{EDD} (\mathcal{R}_{n}^{1,1})$ counts the number of critical points of the following \textit{M\"{o}bius alignment problem}: given $2n$ generic real numbers $X_1, x_1, \ldots , X_n, x_n,$ compute a M\"{o}bius transformation $C$ that minimizes the squared Euclidean error defined by $
( C(X_1) - x_1 )^2 + \ldots + (C(X_n) - x_n )^2.$
Analogous problems have been studied in the computer vision literature~\cite{LeChinSuter}.
Our conjecture implies that the number of solution paths tracked
by homotopy continuation methods for computing the global minimum
can scale linearly in $n$. With a view towards proving this conjecture,~\Cref{thm:sing-res11} above may provide a useful first step.

{\small
\bibliographystyle{alpha}
\bibliography{VisionBib}
}

\end{document}